\documentclass[letter,11pt]{article}

\usepackage{setspace}
\usepackage{subcaption}
\usepackage{enumerate}
\usepackage[T1]{fontenc}
\usepackage{amsmath}
\usepackage{amsfonts}
\usepackage{graphicx}
\usepackage{tensor}
\usepackage{tcolorbox}
\usepackage{amssymb}
\usepackage{xspace}
\usepackage{multirow}
\usepackage{geometry}
\usepackage{soul}
\usepackage{colortbl}
\usepackage{wrapfig}
\usepackage{algorithm}
\usepackage{algorithmicx}
\usepackage{algpseudocode}
\usepackage{amsthm}
\usepackage{enumitem}
\usepackage{todonotes}
\usepackage{todonotes}
\usepackage{mathtools}
\usepackage{mathrsfs}
\usepackage{cite}
\usepackage{tikz}
\usepackage{physics}
\usepackage{complexity}
\usepackage[pdftex, plainpages = false, pdfpagelabels, 
                 bookmarks=false,
                 bookmarksopen = true,
                 bookmarksnumbered = true,
                 breaklinks = true,
                 linktocpage,
                 pagebackref,
                 colorlinks = true,  
                 linkcolor = blue,
                 urlcolor  = blue,
                 citecolor = red,
                 anchorcolor = green,
                 hyperindex = true,
                 hyperfigures
                 ]{hyperref}
\usepackage[capitalise]{cleveref}
\usepackage{thmtools} 
\usepackage{thm-restate}
\usepackage{wrapfig}
\usepackage{bm}
\usepackage{tikz}
\usepackage{svg} 
\usetikzlibrary{decorations.pathreplacing,positioning,shapes.geometric}

\usepackage{bm}
\usepackage{bbm}
\usepackage{thmtools}

\newcommand{\ceil}[1]{\left\lceil #1 \right\rceil}
\newcommand{\floor}[1]{\left\lfloor #1 \right\rfloor}

\newcommand{\set}[1]{\left\{#1\right\}}
\newcommand{\sm}{\setminus}
\newcommand{\mt}{\emptyset}

\newcommand{\nat}{{\mathbb N}}

\usepackage{tikzit}

\tikzstyle{box}=[shape=rectangle, text height=1.5ex, text depth=0.25ex, yshift=0.5mm, fill=white, draw=black, minimum height=5mm, yshift=-0.5mm, minimum width=5mm, font={\small}]
\tikzstyle{gate}=[shape=rectangle, text height=1.5ex, text depth=0.25ex, yshift=0.5mm, fill=white, draw=black, minimum height=5mm, yshift=-0.5mm, minimum width=5mm, font={\small}, tikzit category=circuit]
\tikzstyle{big gate}=[shape=rectangle, text height=1.5ex, text depth=0.25ex, yshift=0.5mm, fill=white, draw=black, minimum height=10mm, yshift=-0.5mm, minimum width=5mm, font={\small}, tikzit category=circuit]
\tikzstyle{Z dot}=[inner sep=0mm, minimum size=2mm, shape=circle, draw=black, fill=white, tikzit category=zx]
\tikzstyle{Z phase dot}=[minimum size=5mm, font={\footnotesize\boldmath}, shape=rectangle, rounded corners=2mm, inner sep=0.2mm, outer sep=-2mm, scale=0.8, tikzit shape=circle, draw=black, fill=white, tikzit draw=blue, tikzit category=zx]
\tikzstyle{X dot}=[Z dot, shape=circle, draw=black, fill={rgb,255: red,255; green,136; blue,136}, tikzit category=zx]
\tikzstyle{X phase dot}=[Z phase dot, tikzit shape=circle, tikzit draw=blue, fill={rgb,255: red,255; green,136; blue,136}, font={\footnotesize\boldmath}, tikzit category=zx]
\tikzstyle{hadamard}=[fill=yellow, draw=black, shape=rectangle, inner sep=0.6mm, minimum height=1.5mm, minimum width=1.5mm, tikzit category=zx]
\tikzstyle{paulibox}=[fill={rgb,255: red,221; green,221; blue,255}, draw=black, shape=rectangle, inner sep=0.6mm, minimum height=5mm, minimum width=5mm, font={\footnotesize}, text height=1.5ex, text depth=0.25ex, tikzit category=zx]
\tikzstyle{vertex}=[inner sep=0mm, minimum size=1mm, shape=circle, draw=black, fill=black, tikzit category=misc]
\tikzstyle{vertex set}=[inner sep=0mm, minimum size=2mm, shape=circle, draw=black, fill=white, font={\footnotesize\boldmath}, tikzit category=misc]
\tikzstyle{small black dot}=[fill=black, draw=black, shape=circle, inner sep=0pt, minimum width=1.2mm, tikzit category=circuit]
\tikzstyle{cnot ctrl}=[fill=black, draw=black, shape=circle, inner sep=0pt, minimum width=1.2mm, tikzit category=circuit]
\tikzstyle{cnot targ}=[fill=white, draw=white, shape=circle, tikzit category=circuit, label={center:$\oplus$}, inner sep=0pt, minimum width=2.1mm, tikzit fill={rgb,255: red,102; green,204; blue,255}, tikzit draw=black]
\tikzstyle{ket}=[fill=white, draw=black, shape=regular polygon, regular polygon sides=3, regular polygon rotate=-30, scale=0.7, inner sep=1pt, tikzit category=circuit, tikzit shape=rectangle, tikzit fill=green]
\tikzstyle{bra}=[fill=white, draw=black, shape=regular polygon, regular polygon sides=3, regular polygon rotate=30, scale=0.7, inner sep=1pt, tikzit category=circuit, tikzit shape=rectangle, tikzit fill=red]
\tikzstyle{scalar}=[shape=rectangle, text height=1.5ex, text depth=0.25ex, yshift=0.5mm, fill=white, draw=black, minimum height=5mm, yshift=-0.5mm, minimum width=5mm, font={\small}]
\tikzstyle{clabel}=[fill=white, draw=none, shape=rectangle, tikzit fill={rgb,255: red,56; green,255; blue,242}, font={\footnotesize}, inner sep=1pt, tikzit category=labels]
\tikzstyle{empty diagram}=[draw={gray!40!white}, dashed, shape=rectangle, minimum width=1cm, minimum height=1cm, tikzit category=misc]
\tikzstyle{white dot}=[Z dot]
\tikzstyle{gray dot}=[X dot]
\tikzstyle{white phase dot}=[Z phase dot]
\tikzstyle{gray phase dot}=[X phase dot]
\tikzstyle{small hadamard}=[hadamard]
\tikzstyle{implies}=[-implies, double, double distance=2pt]
\tikzstyle{blue vertex}=[inner sep=0mm, minimum size=1mm, shape=circle, draw=black, fill={rgb,255: red,24; green,164; blue,239}, tikzit category=misc]
\tikzstyle{green vertex}=[inner sep=0mm, minimum size=1mm, shape=circle, draw=black, fill={rgb,255: red,221; green,255; blue,221}, tikzit category=misc]
\tikzstyle{red vertex}=[inner sep=0mm, minimum size=1mm, shape=circle, draw=black, fill={rgb,255: red,255; green,136; blue,136}, tikzit category=misc]
\tikzstyle{gray vertex}=[inner sep=0mm, minimum size=1mm, shape=circle, draw=black, fill={rgb,255: red,191; green,191; blue,191}, tikzit category=misc]
\tikzstyle{purple vertex}=[inner sep=0mm, minimum size=1mm, shape=circle, draw=black, fill={rgb,255: red,224; green,64; blue,255}, tikzit category=misc]
\tikzstyle{yellow vertex}=[inner sep=0mm, minimum size=1mm, shape=circle, draw=black, fill={rgb,255: red,255; green,241; blue,38}, tikzit category=misc]

\tikzstyle{simple}=[-]
\tikzstyle{hadamard edge}=[-, dashed, dash pattern=on 2pt off 0.5pt, thick, draw={rgb,255: red,68; green,136; blue,255}]
\tikzstyle{box edge}=[-, dashed, dash pattern=on 2pt off 0.75pt, thick, draw={rgb,255: red,128; green,128; blue,128}, fill=none, fill opacity=0.3]
\tikzstyle{brace edge}=[-, tikzit draw=blue, decorate, decoration={brace,amplitude=1mm,raise=-1mm}]
\tikzstyle{diredge}=[->]
\tikzstyle{double edge}=[-, double, shorten <=-1mm, shorten >=-1mm, double distance=2pt]
\tikzstyle{gray edge}=[-, {gray!60!white}]
\tikzstyle{pointer edge}=[->, very thick, gray]
\tikzstyle{boldedge}=[-, line width=2pt, shorten <=-0.17mm, shorten >=-0.17mm]
\tikzstyle{gray fill}=[-, dashed, dash pattern=on 2pt off 0.5pt, thick, draw=none, fill={rgb,255: red,191; green,191; blue,191}, fill opacity=0.3, tikzit draw=black]
\tikzstyle{blue fill}=[-, dashed, dash pattern=on 2pt off 0.5pt, thick, draw=none, fill={rgb,255: red,24; green,164; blue,239}, fill opacity=0.3, tikzit draw=black]
\tikzstyle{green fill}=[-, dashed, dash pattern=on 2pt off 0.5pt, thick, draw=none, fill={rgb,255: red,221; green,255; blue,221}, fill opacity=0.5, tikzit draw=black]
\tikzstyle{purple fill}=[-, dashed, dash pattern=on 2pt off 0.5pt, thick, draw=none, fill={rgb,255: red,224; green,64; blue,255}, fill opacity=0.5, tikzit draw=black]
\tikzstyle{yellow fill}=[-, dashed, dash pattern=on 2pt off 0.5pt, thick, draw=none, fill={rgb,255: red,255; green,241; blue,38}, fill opacity=0.5, tikzit draw=black]
\tikzstyle{red fill}=[-, draw=none, fill={rgb,255: red,255; green,136; blue,136}, fill opacity=0.5, tikzit draw={rgb,255: red,255; green,136; blue,136}]
\tikzstyle{red edge}=[-, draw={rgb,255: red,212; green,0; blue,14}]

\definecolor{zxredfg}{RGB}{100,0,0}
\definecolor{zxgreenfg}{RGB}{0,50,0}

\algtext*{EndWhile}
\algtext*{EndIf}
\algtext*{EndFor}
\algtext*{EndProcedure}
\algtext*{EndFunction}

\newtheorem{lemma}{Lemma}
\newtheorem*{remark}{Remark}
\numberwithin{lemma}{section}
\newtheorem{theorem}[lemma]{Theorem}
\newtheorem{theorem*}[lemma]{Theorem*}
\newtheorem{corollary}[lemma]{Corollary}

\newtheorem{claim}{Claim}[lemma]
\newtheorem{proposition}[lemma]{Proposition}

\newcommand*{\claimproofs}{Proof of the Claim.}
\newenvironment{claimproof}[1][\claimproofs]{\begin{proof}[#1]}{\end{proof}}

\newcommand{\skt}[1]{Squiggly $K_{#1,#1}$\xspace}
\newcommand{\cals}{\mathcal{S}\xspace}
\newcommand{\calc}{\mathcal{C}\xspace}
\newcommand{\calb}{\mathcal{B}\xspace}
\newcommand{\bcalb}{\Bar{\mathcal{B}}\xspace}
\newcommand{\call}{\mathcal{L}\xspace}

\newcommand{\cala}{\mathcal{A}\xspace}
\newcommand{\calf}{\mathcal{F}\xspace}
\newcommand{\abpack}{$(A,B)$-packing\xspace}
\newcommand{\balpack}{$(Y,\frac{1}{2})$-packing\xspace}
\newcommand{\atw}[1]{tw_\alpha(#1)}

\newcommand{\defn}[1]{\textcolor[RGB]{128,0,0}{\emph{#1}}}
\newcommand{\prop}{\mathcal{P}}
\DeclareMathOperator{\dist}{dist}

\newcommand{\tsquig}{\tilde{t}^{\text{\scalebox{0.7}{\tiny\ref*{lem:squigglyktt}}}}}
\newcommand{\pslim}{p^{\text{\scalebox{0.7}{\tiny\ref*{lem:pqslim}}}}}
\newcommand{\qslim}{q^{\text{\scalebox{0.7}{\tiny\ref*{lem:pqslim}}}}}
\newcommand{\dball}{d^{\text{\scalebox{0.7}{\tiny\ref*{prop:coarse_treewidth_our_class}}}}}
\newcommand{\eball}{\epsilon^{\text{\scalebox{0.7}{\tiny\ref*{prop:coarse_treewidth_our_class}}}}}
\newcommand{\kball}{k^{\text{\scalebox{0.7}{\tiny\ref*{prop:coarse_treewidth_our_class}}}}}
\newcommand{\lamone}{\lambda_1^{\text{\scalebox{0.7}{\tiny\ref*{thm:strong_barrier}}}}}
\newcommand{\lamtwo}{\lambda_2^{\text{\scalebox{0.7}{\tiny\ref*{thm:strong_barrier}}}}}
\newcommand{\gcost}{g^{\text{\scalebox{0.7}{\tiny\ref*{prop:bipartite_lemma}}}}}
\newcommand{\ccost}{c^{\text{\scalebox{0.7}{\tiny\ref*{lem:AB_packing_new}}}}}
\newcommand{\gthick}{\tau^{\text{\scalebox{0.7}{\tiny\ref*{lem:AB_packing_new}}}}}
\newcommand{\mucostab}{\mu^{\text{\scalebox{0.7}{\tiny\ref*{lem:AB_packing_new}}}}}
\newcommand{\kcost}{\kappa^{\text{\scalebox{0.7}{\tiny\ref*{lem:Y_packing_new}}}}}
\newcommand{\mucostbal}{\mu^{\text{\scalebox{0.7}{\tiny\ref*{lem:Y_packing_new}}}}}

\title{Induced-Minor-Closed Classes have Linear, Square-Root, or Sub-Polynomial Tree-Independence}
\author{Maria Chudnovsky\thanks{Princeton University, Princeton, NJ, USA. Supported by NSF Grants DMS-2348219 and CCF-2505100,  AFOSR grant FA9550-25-1-0275, and a Guggenheim Fellowship.} \and
Julien Codsi\thanks{Princeton University, Princeton, NJ, USA. Supported by NSF Grant DMS-2348219 and by the Fonds de recherche du Québec via the doctoral research scholarship 321124.} \and
Ajaykrishnan E~S\thanks{Department of Computer Science, University of California Santa Barbara, Santa Barbara, CA, USA. Supported by NSF Grant CCF-2505099.} \and
Daniel Lokshtanov$^{\ddagger}$
}

\date{}

\begin{document}

\pagenumbering{gobble}

\maketitle

\begin{abstract}

An \defn{independent set} in a graph $G$ is a set of pairwise non-adjacent vertices.
A \defn{tree decomposition} of  $G$ is a pair $(T, \chi)$ where $T$ is a tree and $\chi : V(T) \rightarrow 2^{V(G)}$ is a function satisfying the following two axioms: for every edge $uv \in E(G)$ there is an $x \in V(T)$ such that $\{u,v\} \subseteq \chi(x)$, and for every vertex $u \in V(G)$ the set $\{x \in V(T) ~|~ u \in \chi(x)\}$ induces a non-empty and connected subtree of $T$.
The sets $\chi(x)$ for $x \in V(T)$ are called the \defn{ bags} of the tree decomposition.  The \defn{tree-independence} number of $G$ is the minimum taken over all tree decompositions of $G$ of the maximum size of an independent set of the graph induced by a bag of the tree decomposition. A graph $H$ is an \defn{ induced minor} of a graph $G$ if a graph isomorphic to $H$ can be obtained from $G$ by a sequence of vertex deletions and edge contractions.



We prove that for every positive integer $t$ there exists an $\epsilon > 0$ such that every graph $G$ either contains the complete bipartite graph $K_{t,t}$ or the wall $W_{t\times t}$ as an induced minor, or $G$ has tree-independence at most $O(2^{O((\log n)^{1-\epsilon})})$.
This leads to algorithms with running time $2^{2^{O((\log n)^{1-\delta})}} = 2^{n^{o(1)}}$, for $\delta > 0$, for a wide range of problems, including {\sc Independent Set}, {\sc Feedback Vertex Set} and $k$-{\sc Coloring}, on $\{K_{t,t}, W_{t\times t}\}$-induced minor free graphs.
Our result is a substantial generalization of existing bounds for the tree-independence and tree-width on various graph classes, and a partial resolution of the conjecture of Chudnovsky, E S, and Lokshtanov [Arxiv, 2025] that $\{K_{t,t}, W_{t\times t}\}$-induced minor free graphs have poly-logarithmic tree independence number. The generality comes at the cost of a {\em sub-polynomial}, rather than {\em poly-logarithmic} upper bound. 

Our result leads to a complete classification of induced-minor closed classes into ones that have sub-polynomial tree-independence, tree-independence equal to $\sqrt{n}$ up to poly-logarithmic factors, and linear tree-independence. 
\end{abstract}

\newpage
\pagenumbering{arabic}

\section{Introduction}

A \defn{tree decomposition} of a graph $G$ is a pair $(T,\chi)$ in which $T$ is a
tree and $\chi\colon V(T)\to 2^{V(G)}$ is a function subject to two constraints:
every edge of $G$ is contained in $\chi(x)$ for some $x\in V(T)$, and for every
vertex $v\in V(G)$ the set of nodes $x$ with $v\in\chi(x)$ induces a non-empty
connected subtree of $T$.
The sets $\chi(x)$ for $x\in V(T)$ are the \defn{bags} of the decomposition.
Tree decompositions are among the most powerful tools in modern graph theory.
On the structural side they are the backbone of the Graph Minors project of
Robertson and Seymour~\cite{RobertsonS04}, and on the algorithmic side they underlie a vast
body of efficient algorithms, from Courcelle's theorem~\cite{ArnborgLS91, Courcelle90, BoriePT92}, the
isomorphism algorithm of Grohe and Marx~\cite{GroheM15} (see also~\cite{Neuen24}), and the bidimensionality theory
of Demaine et al.~\cite{DemaineFHT05} to the almost-linear time algorithm of Korhonen, Pilipczuk, and Stamoulis  
for minor testing~\cite{KorhonenPS24} (see also~\cite{Korhonen25}) and the linear time algorithm of Korhonen for finding the edge-$k$-connected components~\cite{Korhonen25b}, to name just a few.

The standard way to measure the quality of a tree decomposition $(T,\chi)$ is by
its \defn{width}, which is defined as $\max_{x\in V(T)}|\chi(x)|-1$. The \defn{treewidth}
$\mathrm{tw}(G)$ of $G$ is the minimum width over all its tree decompositions.
For many purposes, however, this measure is overly pessimistic: a tree
decomposition whose bags are large but structurally simple can still be very
useful algorithmically. 
This has prompted the introduction of several alternative width measures, among
them \emph{minor-matching hypertree width}~\cite{yolov2018minor},
the number of bounded-radius balls, stars, or other prescribed objects needed to cover each bag~\cite{chudnovsky2026coarse,DraganA14,GottlobLS02,nguyen2025asymptotic},
and the \emph{tree-independence number}~\cite{dallard2024treewidth, yolov2018minor}.
In this paper we focus on the tree-independence number.
The \defn{independence number} $\alpha(H)$ of a graph $H$ is the
maximum size of a set of pairwise non-adjacent vertices of $H$.
The \defn{tree-independence number} $\mathrm{tw}_\alpha(G)$ of $G$ is the
minimum, over all tree decompositions $(T,\chi)$ of $G$, of
$\max_{x\in V(T)}\alpha\bigl(G[\chi(x)]\bigr)$. That is, rather than charging a
bag its number of vertices, we charge it only the size of the largest
independent set in the bag.

To streamline the discussion, it is convenient to speak of the treewidth and the
tree-independence of an entire \emph{class} of graphs.
For a class $\mathcal H$ we define $\mathbf{tw}^{\mathcal H}\colon\mathbb N\to
\mathbb N$ by letting $\mathbf{tw}^{\mathcal H}(n)$ be the maximum treewidth of a
graph in $\mathcal H$ on at most $n$ vertices, and we define the
\emph{tree-independence} $\mathbf{tw}^{\mathcal H}_\alpha(n)$ of $\mathcal H$
analogously, with $\mathrm{tw}_\alpha$ in place of $\mathrm{tw}$.

The appeal of the tree-independence number stems from two facts.
First, since $\alpha(G[\chi(x)])\le|\chi(x)|$ for every bag $\chi(x)$, it never
exceeds the treewidth by more than one,
yet for many classes it is dramatically smaller: chordal graphs have treewidth
$n-1$ but tree-independence number $1$ \cite{gavril1974intersection}; even-hole-free graphs and
$3$-path-configuration-free graphs have treewidth $n-1$ but tree-independence
number $\log^{O(1)}n$~\cite{chudnovsky2025tree, ChudnovskyHLS26}; and unit disk graphs have treewidth $n-1$ but
tree-independence number $O(\sqrt n)$~\cite{de2020framework, treeIndCMSOAlgFast}.
Second, while a small tree-independence number does not confer the full
algorithmic power of a small treewidth, it still suffices to solve a rich family
of problems efficiently.
Concretely, for every fixed integer $q$ and every fixed $\mathrm{CMSO}$ (counting
monadic second-order) formula $\varphi$, a maximum-weight induced subgraph of
treewidth at most $q$ satisfying $\varphi$ can be found efficiently on graphs of
small tree-independence number~\cite{treeIndCMSOAlgFast} (see also \cite{LimaMMORS24}).
This already captures many classical problems, among them \textsc{Independent
Set} and \textsc{Feedback Vertex Set}, even though it falls short of the full
generality of Courcelle's theorem on graphs of bounded treewidth.
Still other well-studied problems that lie outside this framework are
nonetheless known to be tractable on graphs of small tree-independence number;
for example, $k$-\textsc{Coloring} for every fixed $k$~\cite{yolov2018minor, dallard2024treewidth}.


Treewidth is closed under taking minors: deleting a vertex, deleting an edge, or
contracting an edge never increases it.
This monotonicity motivates a systematic study of how excluding one or more
graphs as a minor constrains the treewidth of a class, and here the picture is
completely understood.
A minor-closed class $\mathcal H$ falls into exactly one of three cases: 
{\em (i)} $\mathcal H$ contains all graphs, in which case $\mathbf{tw}^{\mathcal H}(n)=
\Theta(n)$,
or {\em (ii)} $\mathcal H$ contains the $t$ by $t$ wall for all $t$, but not all graphs, in
which case $\mathbf{tw}^{\mathcal H}(n)=\Theta(\sqrt n)$,
or {\em (iii)} $\mathcal H$ excludes the $t$ by $t$ wall for some $t$, in which case $\mathbf{tw}^{\mathcal H}(n)=
\Theta(1)$. 
Here the \defn{$t$ by $t$ wall}, denoted by $W_{t\times t}$, is defined as the graph with vertex set 
\[
\left\{(x,y) ~:~ 0 \leq x < 2t \mbox{ and } 0 \leq y < t\right\} \setminus \left\{(0,0), (x_t, t-1)\right\}
\]
where $x_t = 0$ if $t$ is even and $2t-1$ if $t$ is odd, and every pair $(x,y)$ and $(x',y')$ of vertices are adjacent if 
$y=y' \mbox{ and } |x-x'|=1$ or if 
$|y-y'|=1 \mbox{, } x=x' \mbox{ and } \max(y,y') \equiv x \mod 2$. (see Figure~\ref{fig:grid})

\begin{figure}[htbp]
    \centering
    \includegraphics[width=0.9\linewidth]{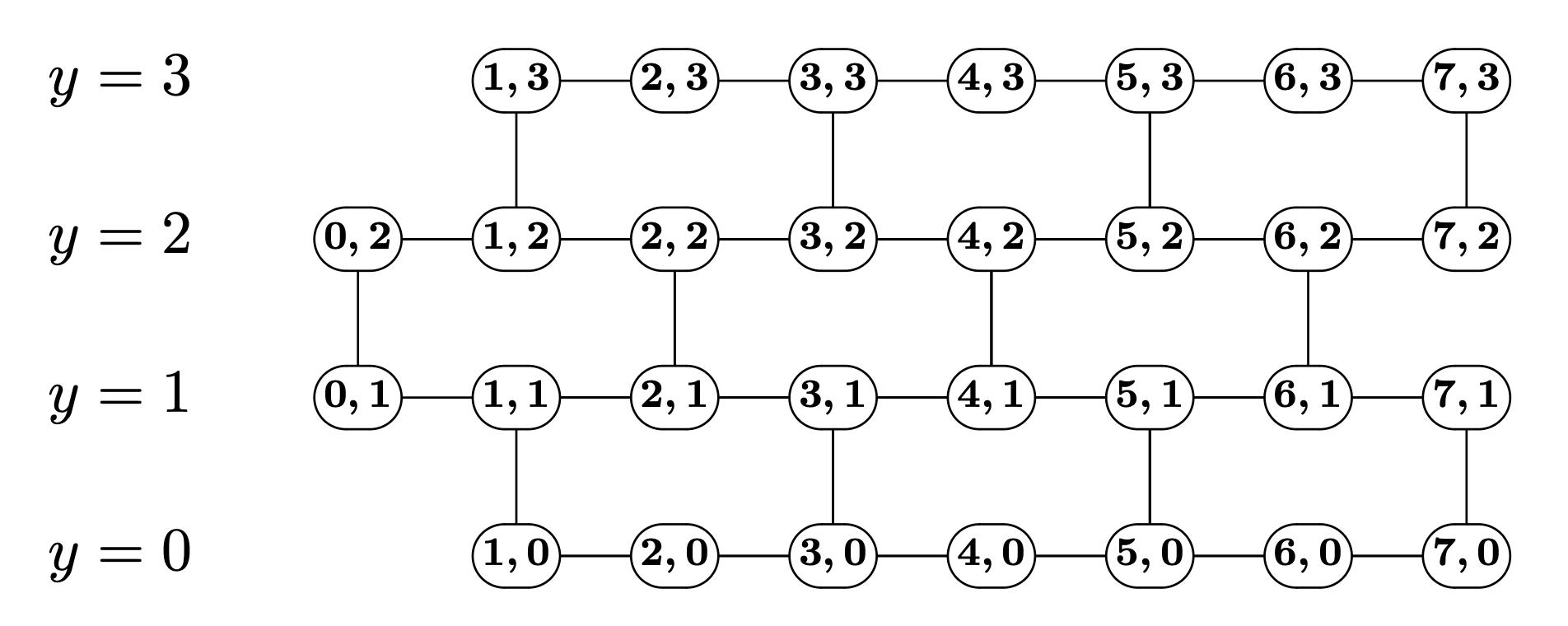}
    \caption{\em The wall \(W_{4\times 4}\) } 
    \label{fig:grid}
\end{figure}

The $O(\sqrt n)$ upper bound in case {\em (ii)} goes back to a classical theorem of Alon, Seymour, and Thomas~\cite{alon1990separator} (see also~\cite{BonnetKLLM26}), who proved that every $n$-vertex graph excluding $K_h$ as a minor has a balanced separator of size $O(h^{3/2}\sqrt n)$, and hence treewidth $O_h(\sqrt n)$. The $t$ by $t$ grids have treewidth $t$~\cite{diestel2025graph}, and therefore show that planar, and therefore $K_5$-minor free graphs, have treewidth $\Omega(\sqrt{n})$. 
The dividing line between case {\em (ii)} and {\em (iii)} follows from the 
Excluded Grid Theorem~\cite{chekuri2016polynomial, chuzhoy2021towards, robertson1986graph}, which states that every graph that excludes $W_{t\times t}$ as a minor has treewidth at most $\tilde{O}(t^9)$.
Here $\tilde{O}$ notation suppresses polylogarithmic factors.

The tree-independence number is closed under \defn{induced minors} --- deleting a
vertex or contracting an edge never increases it --- but, unlike treewidth, it
is \emph{not} closed under taking minors, since deleting an edge \emph{can}
increase it.
This motivates studying how excluding one or more graphs as an induced minor
affects the tree-independence number.
Here, in stark contrast to the clean minor-closed trichotomy for treewidth, the
picture is far from understood.
We summarize the known results most relevant to our work.
\begin{itemize}
    \item If $\mathcal H$ contains the complete bipartite graph $K_{t,t}$ for
    every $t$, then $\mathbf{tw}^{\mathcal H}_\alpha(n)=\Theta(n)$, since the
    tree-independence number of $K_{t,t}$ is $t$~\cite{dallard2024treewidth}.
    \item If $\mathcal H$ contains the $t\times t$ wall $W_{t\times t}$ for every
    $t$, then $\mathbf{tw}^{\mathcal H}_\alpha(n)=\Omega(\sqrt n)$, since the
    tree-independence number of $W_{t\times t}$ is at least $t/2 = \Theta(\sqrt{|V|})$.
    Here the $t/2$ lower bound for the tree-independence of walls follows from the fact that the treewidth of $W_{t\times t}$ is equal to $t$~\cite{diestel2025graph} and that $W_{t\times t}$ is bipartite. 
    
    \item The class of $\{K_{5,5},W_5\}$-induced-minor-free graphs has
    tree-independence number $\Omega(\log n)$~\cite{SintiariT21, chudnovsky2025unavoidable}, so excluding even a small
    complete bipartite graph and a small grid does not force a bounded
    tree-independence number.
    
    \item Despite considerable effort~\cite{chudnovsky2024treeIII,dallard2024treewidthIII,dallard2024treewidth4,hilaire2026treewidthV,hajebi2026tree}, no clean conjecture is known for
    precisely which induced-minor-closed classes have tree-independence number
    $O(1)$; the analogous question of which induced-subgraph-closed classes have
    bounded \emph{treewidth} is known to lack a nice answer~\cite{everyGraphEssential}.
\end{itemize}


Our main theorem shows that, despite this apparent complexity, a surprisingly
clean picture emerges once one is willing to tolerate sub-polynomial slack
between the upper and lower bounds: up to this slack, the tree-independence
number of every induced-minor-closed class is governed by just two obstructions
--- the complete
bipartite graph $K_{t,t}$ and the wall $W_{t\times t}$ --- which place the class
into one of three regimes.

\begin{theorem}\label{thm:mainClassification}
    Let $\mathcal H$ be an induced-minor-closed class of graphs.
    \begin{enumerate}\setlength\itemsep{-.7pt}
        \item\label{itm:subpoly} If there is an integer $t$ with
        $K_{t,t}\notin\mathcal H$ and $W_{t\times t}\notin\mathcal H$, then there is
        an $\epsilon>0$ such that
        $\mathbf{tw}^{\mathcal H}_\alpha(n)\le 2^{O((\log n)^{1-\epsilon})}$.
        \item\label{itm:sqrt} If there is an integer $t$ with
        $K_{t,t}\notin\mathcal H$, but $W_{t\times t}\in\mathcal H$ for every $t$,
        then $\mathbf{tw}^{\mathcal H}_\alpha(n)=\tilde\Theta(\sqrt n)$.
        \item\label{itm:linear} If $K_{t,t}\in\mathcal H$ for every $t$, then
        $\mathbf{tw}^{\mathcal H}_\alpha(n)=\Theta(n)$.
    \end{enumerate}

\end{theorem}

\noindent 
The three cases are exhaustive and mutually exclusive.
An immediate consequence of \Cref{thm:mainClassification} is that the growth rate
of the tree-independence number of an induced-minor-closed class $\mathcal H$,
$\limsup_{n\to\infty}\frac{\log\mathbf{tw}^{\mathcal H}_\alpha(n)}{\log n}$, 
is always equal to $0$, $\tfrac12$, or $1$; no intermediate value can occur.

This trichotomy is, however, considerably coarser than its minor-closed
counterpart for treewidth.
There, each of the three regimes pins down $\mathbf{tw}^{\mathcal H}(n)$ up to a
constant factor, for every class.
Here the sub-polynomial regime~(\ref{itm:subpoly}) is left much wider: a class in
it may have tree-independence number anywhere from $O(1)$ to
$2^{(\log n)^{1-\epsilon}}$, and \Cref{thm:mainClassification} does not
distinguish between these.
Determining the exact tree-independence number within this regime remains open.

Our main technical contribution is the tree independence upper bound underlying Theorem~\ref{thm:mainClassification}, point~(\ref{itm:subpoly}). 
Towards this, for every positive integer $t$, define $\calc_t$ to be the set of all $\{K_{t,t}, W_{t\times t}\}$-induced minor free graphs. We show the following theorem. 
\begin{theorem}\label{thm:atw_bound}
    For every positive integer $t$, there exists a positive integer $\nu'(t)$ and real $\epsilon(t) > 0$ such that, for every graph $G\in\calc_t$ on $n$ vertices,
    \[
        \atw{G} \ \ \leq\ \ 2^{\nu'(t)\log^{1-\epsilon(t)}n}.
    \]
\end{theorem}

Theorem~\ref{thm:atw_bound} is a significant step towards a conjecture of Chudnovsky et al.~\cite{chudnovsky2026} that for every positive integer $t$ there exists a positive integer $d$ such that every graph $G \in \calc_t$ has $\mathrm{tw}_\alpha(G) \leq O((\log n)^d)$.
Together with classic Ramsey bounds for graphs excluding a $K_{t,t}$ as an induced subgraph~\cite{erdos1989ramsey} our result immediately implies that for every integer $t$ there exists an $\epsilon, \delta > 0$ such that for every graph $G$ in $\calc_t$, either $G$ contains a complete graph $K_\omega$ with $\omega \geq 2^{\delta((\log n)^{1-\epsilon})}$ as an induced subgraph or $\mathrm{tw}(G) \leq 2^{O((\log n)^{1-\epsilon})}$. This unifies and generalizes tree-independence and treewidth upper bounds shown in~\cite{dallard2024treewidth, dallard2024treewidth4, ChudnovskyHLS26, chudnovsky2024treeIII, chudnovsky2025tree, chudnovsky2025treeV, chudnovsky2025treeVI, chudnovsky2025treeVii, chudnovsky2024inducedXV, abrishami2023inducedIV, abrishami2022inducedI, walls, subpolytw, korhonen2023grid, hajebi2026tree}, at the cost of sub-polynomial rather than constant or poly-logarithmic bounds. 
For each of the listed papers, to see that our bound generalizes the one obtained there it is sufficient to compare the forbidden structures required for the respective bound with the forbidden structures for $\calc_t$ from Theorem~\ref{thm:CtInTermsOfInducedSubgraphs}.

A few remarks are in order. First, several of the cited papers obtain asymptotically tight bounds in cases where being asymptotically tight matters. Second, our upper bounds only generalize the tree-independence or treewidth bounds proved in the listed papers, many of them also show other results that are of independent interest and are not necessarily covered (even with weaker bounds) by Theorem~\ref{thm:mainClassification}.
Finally, our proof of Theorem~\ref{thm:mainClassification} directly or indirectly invokes the bounds shown in~\cite{korhonen2023grid, subpolytw}, and hence, even though Theorem~\ref{thm:mainClassification} is a generalization of~\cite{korhonen2023grid, subpolytw} it does not give an alternative or independent proof of these two results.  

To complete the trichotomy we also give a new upper bound on the tree-independence number of $K_{t,t}$-induced-minor free graphs. 
\begin{theorem}\label{thm:biclique}
For every positive integer $t$ there exists a positive integer $d$ such that for every $K_{t,t}$-induced-minor free graph $G$, $\mathrm{tw}_\alpha(G) \leq O(\sqrt{n}(\log n)^d)$.
\end{theorem}
Theorem~\ref{thm:biclique} follows quite easily from a combination of known results: that every graph with large tree-independence contains an induced subgraph with large treewidth and poly-logarithmic clique number~\cite{chudnovsky2026}, that every $K_{t,t}$-induced minor free graph with poly-logarithmic clique number has poly-logarithmic average degree~\cite{bourneuf2024polynomial}, and that every $K_{t,t}$-induced minor free graph $G$ has a separator with $O(\sqrt{|E(G)|})$ edges~\cite{korhonen2024induced}.

\paragraph{Algorithmic Consequences.}
Lokshtanov, Pilipczuk and Rzazewski~\cite{treeIndCMSOAlgFast} (see also~\cite{LimaMMORS24}) gave an algorithm for the following problem, called {\sc Max Weight}-$(\phi, r)$-{\sc Induced Subgraph}. 
Here input is a graph $G$, integer $r$ and CMSO$_2$-formula $\phi$, and a weight function $w : V(G) \rightarrow \mathbb{N}$.
The task is to find a vertex subset $S$ such that $\mathrm{tw}(G[S]) \leq r$, $G[S]$ satisfies the property $\phi$, and $w(S) = \sum_{v \in S} w(v)$ is maximized, or determines that no such set $S$ exists. 
We refer to~\cite{treeIndCMSOAlgFast} and references within for a definition of CMSO$_2$.
Here it suffices to know that {\sc Max Weight}-$(\phi, r)$-{\sc Induced Subgraph} generalizes a host of well-known problems, including {\sc Independent Set} and {\sc Feedback Vertex Set}.
The algorithm of Lokshtanov, Pilipczuk and Rzazewski~\cite{treeIndCMSOAlgFast} runs in time  $f(r, \phi)n^{g(r, \phi)k}$ if a tree decomposition with tree-independence $k$ is given as input. 
Additionally, if the tree-independence number of $G$ is at most $k$, a tree-decomposition with tree-independence number at most $8k$ can be computed in time $2^{O(k^2)}n^{O(k)}$ using the algorithm of Dallard et al.~\cite{dallard2025computing} (see also~\cite{yolov2018minor}).

%
%
%
Dallard et al.~\cite{dallard2024treewidth} observed that a $r$-colorable graph with tree-independence $k$ has treewidth at most $rk$.
A tree decomposition of tree-width at most $2rk$ can then be computed in time $2^{O(rk)}n$, using the algorithm of Korhonen~\cite{korhonen2023single}. It is well-known that $r$-{\sc Coloring} can be solved in time $O(r^{k'})n$ if a tree decomposition of tree-width $k'$ is given as input~\cite{cygan2015parameterized}.
Together with our upper bound on the tree-independence of graphs in $\calc_t$, the algorithms above yield the following theorem. 

\begin{theorem}\label{thm:mainAlgorithmic}
For every CMSO$_2$ formula $\phi$ and integers $t$, $r$ there exists an $\epsilon > 0$  such that 
{\sc Max Weight}-$(\phi, r)$-{\sc Induced Subgraph} and $r$-{\sc Coloring} have algorithms with running time
$2^{2^{O((\log n)^{1-\epsilon})}}$
on graphs in $\calc_t$.
\end{theorem}

The running time of the algorithms of Theorem~\ref{thm:mainAlgorithmic} is strongly sub-exponential in the sense that it grows slower than $2^{n^\delta}$ for {\em every} $\delta > 0$. Thus, if any of the problems covered by Theorem~\ref{thm:mainAlgorithmic}, including {\sc Independent Set} and {\sc Feedback Vertex Set}, are \textsf{NP}-hard on $\{K_{t,t}, W_{t\times t}\}$-induced-minor free graphs for some fixed integer $t$, then every problem in \textsf{NP} admits algorithms with running time $O(2^{n^\delta})$ for every $\delta > 0$, a complexity theoretic collapse almost as dramatic as \textsf{P} =  \textsf{NP}. 
We remark that the existence of a {\em polynomial} time algorithm for {\sc Independent Set} on even-hole-free graphs, a class much less general than $\calc_t$, remains a prominent open problem~\cite{adler2017rank,cameron2018structure,conforti2015stable,husic2019independent,le2018coloring}. 

\paragraph{Proof Overview.}
We give a brief outline of the proof. As the proof relies on ideas from multiple previous works, the overview is interleaved with a brief summary of the most relevant existing literature.
Gartland and Lokshtanov~\cite{gartland2020independent} introduced the ``layered sets'' technique in the context of branching algorithms for the {\sc Independent Set} problem. 
Chudnovsky et al.~\cite{ChudnovskyHLS26} observed that the layered sets technique can be used to prove structural bounds for the tree-independence of certain, quite restricted, graph classes.
Subsequently, Chudnovsky et al.~\cite{chudnovsky2025treeV} showed that the structural lemma underlying the layered set technique to bound tree-independence holds in $\calc_t$ for all $t$. In other words, the layered sets technique can be used in all induced-minor closed classes where one could hope to prove sub-polynomial upper bounds on tree-independence. 
In a recent paper, Chudnovsky et al.~\cite{chudnovsky2026coarseLayeredSets} proved that on graphs in $\calc_t$ that do not have large cliques (cliques of poly-logarithmic size are ok), the structural lemma underlying the layered set technique can be generalized to constant radius balls (see Proposition~\ref{r distance lemma}).

The restriction to graphs with poly-logarithmic clique number required for Proposition~\ref{r distance lemma} to work might seem to limit its applicability for proving tree-independence bounds for graphs in $\calc_t$ with arbitrary clique number. However, Chudnovsky et al.~\cite{chudnovsky2026} show that in order to prove an upper bound for the tree-independence number of graphs in $\calc_t$, it is sufficient to upper bound the tree-independence number (or treewidth) of graphs in $\calc_t$ with poly-logarithmic clique number (see Proposition~\ref{prop:tw_and_talpha}). 
In light of this, Chudnovsky et al.~\cite{subpolytw} gave a sub-polynomial $2^{O((\log n)^{1-\epsilon})}$ upper bounds on the tree-width of graphs in $\calc_t$ with {\em constant} clique number. In an attempt to bridge the gap between constant and poly-logarithmic clique number, Chudnovsky et al.~\cite{coarsetw} show that for every integer $t$ there exists an $\epsilon > 0$ such that every graph in $\calc_t$ has a tree decomposition where every bag can be covered by $2^{O((\log n)^{1-\epsilon})}$ balls of radius at most $8$ (see Proposition~\ref{prop:coarse_treewidth_our_class}). The proof of this result is a direct invocation of the main result of~\cite{subpolytw} coupled with a graph partitioning scheme inspired by Scott~\cite{ScottLayerings} and Klein et al.~\cite{klein1993excluded}.

The main reason that the treewidth bound of Chudnovsky et al.~\cite{subpolytw} only works for graphs with {\em constant}, rather than poly-logarithmic, clique number is that Chudnovsky et al.~\cite{subpolytw} rely on a recursive contraction scheme pioneered by Korhonen~\cite{korhonen2023grid} (see also~\cite{bonnet2023treewidth}). The contraction scheme makes the $\epsilon$ in the $2^{O((\log n)^{1-\epsilon})}$ upper bound drop at an (at least) exponential rate as a function of the clique number, rendering the method toothless already for graphs with clique size $\log\log n$. The contraction scheme is used by Chudnovsky et al.~\cite{subpolytw} to (essentially) guarantee that the graph $G$ under consideration has a tree decomposition in which every bag is covered by $2^{O((\log n)^{1-\delta})}$ stars (i.e. balls of radius $1$). This tree decomposition is used in conjunction with the layered sets technique to give a tree decomposition of $G$ of width at most $2^{O((\log n)^{1-\epsilon})}$, where $\epsilon$ is lower bounded in terms of $t$ and $\delta$. 

The very short summary of the proof of this paper is that we work with graphs in $\calc_t$ with poly-logarithmic clique number, and carry through the proof of Chudnovsky et al.~\cite{subpolytw}, but avoid the recursive contraction scheme which causes the exponential decay in $\epsilon$.
The recursive contraction scheme was used to guarantee a tree decomposition in which every bag is covered by $2^{O((\log n)^{1-\delta})}$ balls of radius $1$. Instead we invoke Chudnovsky et al.~\cite{coarsetw} (Proposition~\ref{prop:coarse_treewidth_our_class}) to obtain a tree decomposition where every bag can be covered by $2^{O((\log n)^{1-\delta})}$ balls of radius at most $8$.
To turn this tree decomposition into one where every bag has $2^{O((\log n)^{1-\epsilon})}$ {\em vertices}, we use the layered set technique for constant radius balls (Section~\ref{sec:layer}), powered by the recent result of Chudnovsky et al.~\cite{chudnovsky2026coarseLayeredSets} (Proposition~\ref{r distance lemma}).

Several technical complications occur, and in order to overcome them we prove and use two new approximate characterizations of the class $\calc_t$. The first (see Lemma~\ref{lem:squigglyktt}) is that for every positive integer $t$ there exists a $t'$ such that for every graph $G$ in $\calc_t$ the graph $G^\circ$ does not contain an induced $K_{t',t'}$. Here $G^\circ$ is the graph with a vertex for every connected vertex set in $G$, and two vertices in $G^\circ$ are adjacent if the corresponding sets in $G$ have non-empty intersection or an edge between them. 
This characterization in terms of $G^\circ$ is interesting in its own right, as Dallard et al.~\cite{dallard2024treewidth} have previously observed that $\atw{G^\circ} = \atw{G}$.
The second (see Lemma~\ref{lem:pqslim}) is that for every $t$ there exist integers $p$, $q$, such that for every graph $G$ in $\calc_t$ there does not exist a collection of $p$ pairwise disjoint and anti-complete sets $S_1, \ldots, S_p$, such that for every pair $S_i$, $S_j$, there are at least $q$ internally disjoint and anti-complete paths from $S_i$ to $S_j$.


\paragraph{Paper Organization.}
\cref{sec:prelim} contains notation and preliminaries. \cref{sec:approxChar} proves the approximate characterizations of \(\calc_t\) used later: first via ``squiggly complete bipartite graphs'', and then via ``strong slimness''. \cref{sec:layer} develops the bounded-radius layered-sets
lemma and specializes it to balanced separators and \(A\)--\(B\) separators. \cref{sec:barrier} then proves the existence of a combinatorial object in the class $\calc_t$ that we need, called \emph{strong-barriers}. \cref{sec:absep} then uses \emph{strong-barriers} to show that pairs of vertex subsets that are ``slim'' can be easily separated in $\calc_t$. \cref{sec:mainthm} consolidates the results from the previous sections to prove \cref{thm:atw_bound}, and \cref{sec:8} contains the proof of \cref{thm:biclique}. We conclude the paper with~\cref{sec:conclusion}.

 \section{Preliminaries and notation}\label{sec:prelim}
 
\medskip


All graphs in this paper are finite and simple. If \(X\subseteq V(G)\), then
\(G[X]\) denotes the subgraph of \(G\) induced by \(X\), and \(G\setminus X\)
denotes \(G[V(G)\setminus X]\). If \(H\) is an induced subgraph of \(G\), we
write \(H\subseteq G\). In this paper, we use induced subgraphs and their vertex sets interchangeably.
For \(X,Y\subseteq V(G)\), we say that \(X\) and \(Y\)
are \defn{anticomplete} if there is no edge of \(G\) with one end in \(X\) and
one end in \(Y\). We say that \(X\) and \(Y\) \defn{touch} if
\(N_G[X]\cap Y\neq\emptyset\). All logarithms are base \(2\).

For \(v\in V(G)\) and an integer \(r\geq 0\), let \(N^r_G[v]\) denote the set
of vertices at distance at most \(r\) from \(v\) in \(G\). For
\(X\subseteq V(G)\), set
\[
N^r_G[X]=\bigcup_{x\in X}N^r_G[x].
\]
We also write \(B_r(X,G)\) for \(N^r_G[X]\). If \(X,Y\subseteq V(G)\), then
\(\dist_G(X,Y)\) is the minimum length of a path in \(G\) with one end in \(X\)
and one end in \(Y\), with value \(\infty\) if no such path exists.

For a graph class \(\mathcal H\), we write \(\mathbf{tw}^{\mathcal H}_\alpha(n)\) for
the maximum value of \(\atw{G}\) over all graphs \(G\in\mathcal H\) with at
most \(n\) vertices.

For a positive integer \(t\), we denote by \defn{\(\mathcal F_t\)} the class of
\(K_{t,t}\)-induced-minor-free graphs. Recall that \defn{\(\calc_t\)} denotes the
class of graphs with no induced minor isomorphic to \(K_{t,t}\) or \(W_{t\times t}\). Thus
\(\calc_t\subseteq \mathcal F_t\). Both classes are hereditary.

Let \(G\) be a graph and let \(S\subseteq V(G)\). We say that \(\hat S\) is an
\defn{\(r\)-cover} of \(S\) in \(G\) if \(S\subseteq N^r_G[\hat S]\). The set
\(S\) is \defn{\((k,r)\)-coverable} in \(G\) if it has an \(r\)-cover
\(\hat S\) in \(G\) with \(|\hat S|\leq k\). We call such a set \(\hat S\) a
\defn{\((k,r)\)-cover} of \(S\). We say that \(S\) is \defn{internally
\((k,r)\)-coverable} in \(G\) if it has a \((k,r)\)-cover contained in \(S\).

Let \(w:V(G)\to \mathbb R_{\geq 0}\). For \(X\subseteq V(G)\), write
\(w(X)=\sum_{v\in X}w(v)\). We call \(w\) a \defn{weight function} if
\(w(V(G))\leq 1\). A graph \(G\) is \defn{\((w,\varepsilon)\)-balanced} if
every connected component \(C\) of \(G\) satisfies \(w(C)\leq\varepsilon\). A
set \(X\subseteq V(G)\) is a \defn{\((w,\varepsilon)\)-balanced separator} if
\(G\setminus X\) is \((w,\varepsilon)\)-balanced.

As a slight abuse of notation, for a set $Y\subseteq V(G)$ we may denote by $Y$ the weight function which gives weight $1/|Y|$ (if $Y\neq \mt$) to all vertices in $Y$ and $0$ elsewhere. Therefore, for \(Y\subseteq V(G)\) and \(\varepsilon\in(0,1)\), a set \(X\subseteq V(G)\)
is a \defn{\((Y,\varepsilon)\)-balanced separator} if every connected component
of \(G\setminus X\) contains at most \(\varepsilon |Y|\) vertices of \(Y\). 

We use the following standard fact (see, for example~\cite{diestel2025graph}): 

\begin{proposition}\label{prop: central bag} Let $G$ be a graph and $w$ be a weight function on $G$. If $(T,\chi)$ is a tree decomposition of $G$, then some bag $\chi(x)$ is a $(w,\frac12)$-balanced separator in $G$.
\end{proposition}


\section{Approximate Characterizations of $\calc_t$.}\label{sec:approxChar}
When comparing the class $\calc_t$ against other classes it is often more convenient to use an approximate characterization of $\calc_t$ in terms of forbidden induced subgraphs, rather than forbidden induced minors. We begin this section by deriving such a classification.

Let $s,l$ be positive integers. An \defn{$(s,l)$-constellation} is a graph $C$ with an independent set $S_C$ of size $s$ such that $C-S_C$ has exactly $l$ components, each of which is a path that contains at least one neighbor of every vertex in $S_C$. Observe that an $(s,l)$-constellation forms a $K_{s,l}$-induced-minor.
We need the following, which follows immediately from the main result in \cite{tw16} by Chudnovsky, Hajebi, and Spirkl.

\begin{proposition}\label{prop:constellation}
    For all positive integers $s,l,r$, there is a positive integer $a=a(s,l,r)$ with the following property. 
    Let $G$ be a graph that contains a $K_{a,a}$-induced-minor. Then one of the following holds.
    \begin{itemize}
        \item There is an induced subgraph of $G$ isomorphic to either $K_{r,r}$, a subdivision of $W_{r\times r}$, or the line graph of a subdivision of $W_{r\times r}$.
        \item There is an $(s,l)$-constellation in $G$. 
    \end{itemize}
\end{proposition}

We also need the following theorem of Aboulker, Adler, Kim, Sintiari, and Trotignon~\cite{walls}, which allows us to translate wall-induced-minors into induced subgraphs.
\begin{proposition}\label{prop:wall}
    For all positive integers $t$, there is a  positive integer $b=b(t)$ with the following property. 
    Let $G$ be a graph  that contains a $W_{b\times b}$-induced-minor. Then $G$ contains as an induced subgraph either a subdivision of $W_{t\times t}$, or the line graph of a subdivision of $W_{t\times t}$.
\end{proposition}

Combining Propositions~\ref{prop:wall} and~\ref{prop:constellation} immediately yields the following approximate classification of $\calc_t$.

\begin{theorem}\label{thm:CtInTermsOfInducedSubgraphs}
For every positive integer $r$ there exists a positive integer $t$ such that if a graph $G$ contains $K_{t,t}$ or $W_{t\times t}$ as an induced minor, then $G$ contains as an induced subgraph a subdivision of the wall $W_{r\times r}$, or the line graph of a subdivision of $W_{r\times r}$, or $K_{r,r}$ or an $(r,r)$-constellation.
\end{theorem}

\begin{proof}
Given $r$ set $t = \max( a(r, r, r), b(r) )$ where $a$ is the function guaranteed by Proposition~\ref{prop:constellation}, while $b$ is the function guaranteed by Proposition~\ref{prop:wall}. 

Suppose now that $G$ contains $K_{t,t}$ or $W_{t\times t}$ as an induced minor. 
If $G$ contains $K_{t,t}$ as an induced minor then, by \cref{prop:constellation}, $G$ contains an induced subgraph isomorphic to $K_{r,r}$, a subdivision of $W_{r\times r}$ or the line graph of a subdivision of $W_{r\times r}$, or $G$ contains an $(r, r)$-constellation.
If $G$ contains $W_{t\times t}$ as an induced minor then, by \cref{prop:wall}, $G$ contains as an induced subgraph either a subdivision of $W_{r\times r}$, or the line graph of a subdivision of $W_{r\times r}$.
\end{proof}

Theorem~\ref{thm:CtInTermsOfInducedSubgraphs} is an approximate characterization of $\calc_t$ in the sense that a graph $G$ that contains one of the obstructions guaranteed by Theorem~\ref{thm:CtInTermsOfInducedSubgraphs} with parameter $r$ is not in $\calc_r$. On the other hand, any graph not in $\calc_t$ (where $t$ is the integer guaranteed by the theorem) contains such obstructions. Thus Theorem~\ref{thm:CtInTermsOfInducedSubgraphs} shows that for any $r$ there exists a $t$ such that every graph $G$ not in $\calc_t$ contains a ``nice'' obstruction witnessing that $G$ is not in $\calc_r$.

Now we obtain two more characterizations of $\calc_t$.
Recall that two vertices $u,v$ of a graph are called \defn{true-twins} if $N[u]=N[v]$.
Let $G$ be a graph. 
We define \defn{$\mathsf{twin}(G)$} to be the graph obtained from $G$ by adding a true-twin for each vertex of the graph. 
That is, for each $v\in V(G)$, we add a new vertex $v^\star$, where $v^\star$ is adjacent to every $u\in N_G[v]$ and two new vertices $u^\star$ and $v^\star$ are adjacent if and only if $uv\in E(G)$. 
For a vertex subset $S$, we use $S^\star$ to denote $\{v^\star~:~v\in S\}$.
We remark that none of $K_{t,t}$, a subdivision of $W_{t\times t}$, the line graph of a subdivision of $W_{t\times t}$, or a $(t,t)$-constellation contains a pair of vertices that are true-twins.

%
A \defn{\skt{t}} in $G$ is a family $\{A_1,\dots,A_t\}\cup\{B_1,\dots,B_t\}$ of vertex subsets of $G$ such that:
\begin{itemize}
    \item[--] $A_i$ and $B_i$ induce connected subgraphs for every $i\in[t]$;
    \item[--] $A_i$ and $A_j$ do not touch, and $B_i$ and $B_j$ do not touch, for every distinct $i,j\in[t]$;
    \item[--] $A_i$ and $B_j$ touch for every $i,j\in[t]$.
\end{itemize}
We say that a graph \defn{has a \skt{t}} if it has a collection of vertex subsets which forms an \skt{t}. 
Now we are ready to state and prove the first main lemma of this section.

\begin{lemma}\label{lem:squigglyktt}
    Corresponding to every positive integer $t$, there exists a positive integer $\tilde{t} = \tilde{t}(t)$ such that no graph $G\in\calc_t$ has a \skt{\tilde{t}}.
\end{lemma}
\begin{proof}
    Set $\tilde{t} := a(t,t,t)$, where $a$ is as in \Cref{prop:constellation}.
    Suppose for contradiction that $G\in\calc_t$ has a \skt{\tilde{t}}, witnessed by $\{A_1,\dots,A_{\tilde{t}}\}\cup\{B_1,\dots,B_{\tilde{t}}\}$.
    Let $G' := \mathsf{twin}(G)$ and observe that by construction each $G'-V(G)$ is isomorphic to $G$. 
    So the sets $B_1^\star,\dots,B_{\tilde{t}}^\star$ are connected and pairwise non-touching.
    Furthermore, since each $A_i$ touches each $B_j$, we have that $N(A_i)\cap B_j^\star$ is non-empty for every $i,j\in[\tilde{t}]$.
    Consequently $\{A_1,\dots,A_{\tilde{t}}\}\cup\{B_1^\star,\dots,B_{\tilde{t}}^\star\}$ is the induced minor model of a $K_{\tilde{t},\tilde{t}}$ in $G'$.
    
    By \Cref{prop:constellation}, $G'$ contains an induced subgraph $H$ isomorphic to one of: $K_{t,t}$, a subdivision of $W_{t\times t}$, the line graph of a subdivision of $W_{t\times t}$, or a $(t,t)$-constellation.
    In all cases, no two vertices of $H$ are true-twins in $G'$, and so $H$ is isomorphic to an induced subgraph of $G$. 
    But then, in all cases we conclude that $G$ contains the induced-minor $W_{t \times t}$-model, or an induced $K_{t,t}$-model, contrary to the fact that  $G \in \mathcal{C}_t$.
\end{proof}

\begin{remark}
    For a positive integer $t$, we use $\tsquig(t)$ to denote the constant implied by \Cref{lem:squigglyktt} and when $t$ is clear from context, we simply write $\tsquig$.
\end{remark}

For the final characterization, let $p,q$ be positive integers and let $G$ be a graph.
Two vertex sets $A,B\subseteq V(G)$ are \defn{$q$-wide} in $G$ if $G$ contains a set $\mathcal{P}$ of at least $q$ many $A$--$B$ paths whose interiors are pairwise vertex-disjoint and anticomplete.
Otherwise they are called \defn{ $q$-slim} in $G$.
A graph $G$ is \defn{$(p,q)$-slim} if for every independent set $S$ of cardinality $p$, there exist vertices $u,v\in S$ such that the pair $\{u\},\{v\}$ is $q$-slim in $G$. 
We say that a hereditary graph class is \defn{$(p,q)$-slim}, if every graph in that class is.
We need the following theorem by Chudnovsky, Codsi, Fischer, and Lokshtanov~\cite{subpolytw}.

\begin{proposition} \label{prop:slim}
    For every positive integer $t$, there exist positive integers $p = p(t)$ and $q = q(t)$ such that $\calc_t$ is $(p,q)$-slim.
\end{proposition}

$G$ is said to be \defn{strongly-$(p,q)$-slim} if for every set $\{S_1,S_2,\dots,S_p\}$ of pairwise disjoint and anticomplete sets of cardinality $p$ with the property that $G[S_i]$ is connected for each $i\in[p]$, there exists at least one pair $S_i,S_{j}$ that is $q$-slim in $G$. 
Finally, we say that a hereditary graph class is \defn{strongly-$(p,q)$-slim}, if every graph in that class is.
We now state and prove the second main lemma of this section.

\begin{lemma}\label{lem:pqslim}
    For every positive integer $t\ge 2$, there exist positive integers $p = p(t)$ and $q = q(t)$ such that $\calc_t$ is strongly-$(p,q)$-slim.
\end{lemma}
\begin{proof}
    Let \(a:=a(t,t,t)\) be given by~\cref{prop:constellation} and let \(b:=b(t)\) be given by~\cref{prop:wall}. 
    Apply Proposition~\ref{prop:slim} with parameter \(\max\{a,b\}\), and let \(p,q\) be the resulting integers.
    Let $\{S_1,\dots,S_p\}$ be pairwise disjoint, anticomplete, connected sets in $G\in\calc_t$.
    Let $G' := \mathsf{twin}(G)$ and define $G''$ to be the graph obtained from $G'$ by contracting each $S_i\subseteq V(G')$ to a single vertex $s_i$.
    Since the $S_i$ are pairwise disjoint and anticomplete in $G$, $\hat{S} = \{s_1,\dots,s_p\}$ is an independent set of size $p$ in $G''$.
    Now, suppose for contradiction that no pair $S_i,S_j$ is $q$-slim in $G$.

    \begin{claim}
        If $s_i,s_j\in \hat{S}$ then the pair $\{s_i\}, \{s_j\}$ is $q$-wide in $G''$.
    \end{claim}
    \begin{claimproof}
        Let $s_i,s_j\in \hat{S}$. 
        Since $S_i,S_j$ are $q$-wide in $G$, there exists a set $\mathcal{P}$ of $q$ many $S_i$--$S_j$ paths in $G$ that are internally vertex-disjoint and internally anticomplete.
        Assume without loss of generality that the internal vertices of every path in $P$ lie outside $S_i\cup S_j$.
        For every $P\in\mathcal{P}$, define $P'$ to be the subgraph of $G'$ obtained by replacing every internal vertex $v\in V(P)\cap\bigcup_{s\in[p]} S_s$ with $v^\star$, and observe that $P'$ is an $S_i$--$S_j$ path in $G'$. 
        Since the paths in $\mathcal{P}$ are internally vertex-disjoint and internally anticomplete, $\mathcal{P}' := \{P'~:~P\in\mathcal{P}\}$ is a set of $S_i$--$S_j$ paths in $G'$ with no internal vertices from $\bigcup_{i\in[p]}S_i$, that also have the same properties.
        For every $P'\in\mathcal{P}'$, let $P''$ be obtained from $P'$ by replacing its endpoints with $s_i$ and $s_j$. 
        Since $P'$ has no internal vertex in $\bigcup_{s\in[p]}S_s$, each $P''$ is an $s_i$--$s_j$ path in $G''$. 
        Hence $\mathcal{P}'' := \{P''~:~P'\in\mathcal{P}'\}$ is internally vertex-disjoint and internally anticomplete, witnessing that $\{s_i\},\{s_j\}$ is $q$-wide in $G''$.
    \end{claimproof}
    
    By the contrapositive of \Cref{prop:slim}, $G''$ must contain a $K_{a,a}$ or $W_{b\times b}$ as an induced minor, witnessed by some induced subgraph $H''\subseteq G''$.
    But then, the induced subgraph $H'$ of $G'$ obtained by replacing any $s_i\in V(H'')\cap \hat{S}$ with $S_i\subseteq V(G')$ witnesses a $K_{a,a}$ or $W_{b\times b}$ as an induced minor of $G'$.
    Hence, from \Cref{prop:constellation} and \Cref{prop:wall} we conclude that, $G'$ contains an induced subgraph $H$ isomorphic to one of: $K_{t,t}$, a subdivision of $W_{t\times t}$, the line graph of a subdivision of $W_{t\times t}$, or a $(t,t)$-constellation.
    In all cases, no two vertices of $H$ are true-twins in $G'$, and so $H$ is isomorphic to an induced subgraph of $G$. 
    But then, in all cases, we obtain an induced-minor model of $W_{t \times t}$ or $K_{t,t}$ in $G$, contradicting the fact that $G\in\mathcal{C}_t$. 
\end{proof}

\begin{remark}
    For a positive integer $t$, we use $\pslim(t)$ and $\qslim(t)$ to denote the constant implied by \Cref{lem:pqslim}. When $t$ is clear from context, we simply write $\pslim, \qslim$.
\end{remark}

\section{Layered Sets with Balls of Radius $r$.}\label{sec:layer}

In this section, we prove a bounded-radius version of the layered-sets theorem from \cite{chudnovsky2025treeV}. Informally, given a hereditary property \(\mathcal P\), we show that if every induced subgraph can be made to satisfy \(\mathcal P\) by deleting a set coverable by few bounded-radius balls, then either the whole graph can be made to satisfy \(\mathcal P\) by deleting a bounded set, or we can find many pairwise anticomplete ``small'' sets whose deletion makes the remaining graph satisfy \(\mathcal P\).
The proof iteratively builds a sequence of easily coverable sets whose deletion gives \(\mathcal P\). At each step, an exceptional set is removed so that the level sets, which record how often vertices are covered by previously chosen sets, shrink geometrically. This step crucially uses a recent result from \cite{chudnovsky2026coarseLayeredSets} to ensure that the process is not too costly. We then extract a pairwise anticomplete subsequence.
Finally, we apply the general lemma to two hereditary properties: being balanced with respect to a weight function, and being \(A\)--\(B\) separated. These applications use a good tree decomposition of graphs in \(\mathcal C_t\) from \cite{coarsetw}.

For a sequence $\vec{S} = S_1,\dots, S_k$, we denote by $\vec{S}_{\leq j}$ the subsequence $S_1,\dots,S_j$. We define $S_{\leq 0 }=\mt$.
Let $\mathcal{P}$ be a hereditary graph property. For a graph $G$ and $X\subseteq G$ say that $G$ is \defn{$X$-close} to $\prop$ if $G\sm X \in \prop$. We say that a graph is \defn{$(k,r,\prop)$-close} if there exists a $(k,r)$-coverable set $X$ such that $G$ is $X$-close to $\prop$. We say that a graph family $\mathcal{F}$ is \defn{$(k,r,\prop)$-close} if it is hereditary and every graph in $\mathcal{F}$ is $(k,r,\prop)$-close. 
As a slight abuse of notation, we extend the notion of hereditary properties to labeled graphs.

Given a graph $G$ and a sequence $\vec{S} = (S_1, \ldots, S_\ell)$ of vertex sets in $G$ and integer $i$, we define the \defn{$i$th level set} in $G$ to be $$L_i(G, \vec{S}) = \set{v \in V(G) ~:~ \abs{\set{S_j \in \vec{S} ~:~ v \in S_j}} \geq i }$$
(note that $L_0(G,S)=V(G)$).
For a graph $G$, tuple $\vec{S} = (S_1, \ldots, S_\ell)$ of sets, and positive integer $a$ and $r$, we say that a vertex set $C \subseteq V(G)$ is \defn{$(a,r)$-clean} if for every $i\geq 1$ and vertex $v \in C$ it holds that 
$$\abs{L_i(N^r_C[v],\vec{S})}\quad  \leq\quad \frac{|G|}{2^i}\frac{1}{a}\mbox{.}$$

A \defn{$(k,r,a)$-cleaning sequence} in a graph $G$ consists of three sequences of sets $(S_1, \ldots, S_\ell)$, $(\hat{S}_1, \ldots, \hat{S}_\ell)$, and $(G_1, \ldots, G_\ell)$ such that :
\begin{enumerate}
    \item For every $1\le i\le j\le \ell$, $G_j\subseteq G_i\subseteq G$.
    \item For every \(1\le i\le \ell\), \(\hat S_i\) is a \((k,r)\)-cover of \(S_i\) in \(G_i\). Moreover, for every \(1\le i<j\le \ell\), we have \(\hat S_i\cap V(G_j)=\emptyset\).
    \item For every $1\leq i \leq \ell$ the set $G_i$ is $(2ka,r)$-clean with respect to $\vec{S}_{\leq i-1 }$ 
    \item $G_1$ is $(ka,r)$-clean with respect to $G$  (viewed as a one-element sequence)
    \item $\ell \leq a$,
\end{enumerate}
We call $\ell$ the \defn{length} of the $(k,r,a)$-cleaning sequence.

\begin{lemma}[Closure under subsequences]

Let $(S_1, \ldots, S_\ell)$, $(\hat{S}_1, \ldots, \hat{S}_\ell)$, and $(G_1, \ldots, G_\ell)$ be a $(k,r,a)$-cleaning sequence in a graph $G$.
Let $1 \le i_1 < i_2 < \cdots < i_p \le \ell$ be an increasing sequence of indices.
Let $
\vec{S}' =(S_{i_1},\dots,S_{i_p}),
\vec{\hat S}' = (\hat S_{i_1},\dots,\hat S_{i_p}),
\vec{G}' = (G_{i_1},\dots,G_{i_p})
$ then $\vec{S}'$, $\vec{\hat S}'$ $\vec{G}'$ is a $(k,r,a)$-cleaning sequence.
\end{lemma}

\begin{proof}
Condition 3 holds since for every $t \le p$, $b\geq 1$, and every $v \in G_{i_t}$ we have
$\vec S'_{\le t-1} \subseteq \vec S_{\le i_t - 1}$, and therefore
\[
L_b\!\left(N_{G_{i_t}}^{r}[v], \vec S'_{\le t-1}\right)
\quad \subseteq\quad
L_b\!\left(N_{G_{i_t}}^{r}[v], \vec S_{\le i_t - 1}\right).
\]
Hence,
\[
\left|L_b\left(N_{G_{i_t}}^{r}[v], \vec S'_{\le t-1}\right)\right| \quad
\le \quad
\left|L_b\left(N_{G_{i_t}}^{r}[v], \vec S_{\le i_t - 1}\right)\right| \quad
\le \quad 
\frac{|G|}{2^b (2ka)}.
\]

Condition 4 holds since the property is monotone under taking subgraphs, so $G_{i_1}$ is $(ka,r)$-clean with respect to $G$.
All remaining conditions are inherited directly from the original sequence, so
$\vec{S}', \vec{\hat S}', \vec{G}'$ forms a $(k,r,a)$-cleaning sequence.
\end{proof}





\begin{lemma}\label{lem:smallLevels}
Let $G$ be a graph and let $\vec S$, $\vec{\hat S}$, and $\vec G$ be a $(k,r,a)$-cleaning sequence in $G$ of length $\ell$. Then, for all $1\leq i,j\leq \ell$,
\[
\abs{L_i(G_j,\vec S_{\leq j})}
\quad \leq\quad
\frac{|G|}{2^i}\frac{j}{a}.
\]
\end{lemma}

\begin{proof}
First we prove the case $i=1$. Let $1\leq q\leq \ell$. Since $\hat S_q$ is a $(k,r)$-cover of $S_q$ in $G_q$, we have
\[
S_q\subseteq \bigcup_{x\in \hat S_q}N^r_{G_q}[x].
\]
Moreover, $G_q\subseteq G_1$. Hence, for every $x\in \hat S_q$, we have
$N^r_{G_q}[x]\subseteq N^r_{G_1}[x]$. By Condition 4, $G_1$ is $(ka,r)$-clean with respect to $G$, and so

\[
\abs{S_q}
\quad \leq\quad 
\sum_{x\in \hat S_q}\abs{N^r_{G_q}[x]}
\quad \leq\quad 
|\hat S_q|\frac{|G|}{2ka}
\quad \leq\quad 
\frac{|G|}{2a}.
\]
Therefore, for every $1\leq j\leq \ell$,
\[
\abs{L_1(G_j,\vec S_{\leq j})}
\quad \leq\quad 
\sum_{q\leq j}\abs{S_q}
\quad \leq\quad 
\frac{|G|}{2}\frac{j}{a}.
\]

We now prove the statement for $i\geq 2$ by induction on $j$. The case $j=1$ is trivial, since
$L_i(G_1,\vec S_{\leq 1})=\emptyset$ for every $i\geq 2$. Let $j\geq 1$, and suppose that the statement holds for $j$. We prove it for $j+1$. If $i>j+1$, then
$L_i(G_{j+1},\vec S_{\leq j+1})=\emptyset$, so there is nothing to prove. Thus, we may assume that $2\leq i\leq j+1$. We have
\[
\abs{L_i(G_{j+1},\vec S_{\leq j+1})}
\quad \leq\quad 
\abs{L_i(G_j,\vec S_{\leq j})}
\ +\ 
\abs{L_{i-1}(S_{j+1},\vec S_{\leq j})}.
\]
By the induction hypothesis, or trivially if $i>j$,
\[
\abs{L_i(G_j,\vec S_{\leq j})}
\quad \leq\quad 
\frac{|G|}{2^i}\frac{j}{a}.
\]
It remains to bound the second term. Since $\hat S_{j+1}$ is a $(k,r)$-cover of $S_{j+1}$ in $G_{j+1}$,
\[
\abs{L_{i-1}(S_{j+1},\vec S_{\leq j})}
\quad \leq\quad 
\sum_{x\in \hat S_{j+1}}
\abs{L_{i-1}(N^r_{G_{j+1}}[x],\vec S_{\leq j})}.
\]
Since $i\geq 2$, we have $i-1\geq 1$. Also, by Condition 3, $G_{j+1}$ is $(2ka,r)$-clean with respect to $\vec S_{\leq j}$. Hence
\[
\abs{L_{i-1}(N^r_{G_{j+1}}[x],\vec S_{\leq j})}
\quad \leq\quad 
\frac{|G|}{2^{i-1}}\frac{1}{2ka}.
\]
Therefore
\[
\abs{L_{i-1}(S_{j+1},\vec S_{\leq j})}
\quad \leq\quad 
|\hat S_{j+1}|\frac{|G|}{2^{i-1}}\frac{1}{2ka}
\quad \leq\quad 
k\frac{|G|}{2^{i-1}}\frac{1}{2ka}
\quad =\quad 
\frac{|G|}{2^i a}.
\]
Combining both inequalities yields
\[
\abs{L_i(G_{j+1},\vec S_{\leq j+1})}
\quad \leq\quad 
\frac{|G|}{2^i}\frac{j}{a}
\;+\;
\frac{|G|}{2^i a}
\quad =\quad 
\frac{|G|}{2^i}\frac{j+1}{a},
\]
as desired.
\end{proof}

\begin{lemma}\label{lemma: layered sets disjoint cores}
    Let $a,k,n,p,r$ be positive integers such that $a> pk\ceil{\log n}$ and $n\geq 2$.
    Let $G$ be a $n$-vertex graph and $\vec{S}$, $\vec{\hat{S}}$, and $\vec{G}$ be a $(k,r,a)$-cleaning sequence in $G$ of length $pk\ceil{\log n}$. Then there exists a subsequence ${i_1},\dots, {i_p}$ such that $\hat S_{i_e}\cap S_{i_f}= \mt$ for every $e\neq f$.
\end{lemma}

\begin{proof}
    By induction on $p$. If $p=1$ then the statement is trivially true. Let $\ell = pk \ceil{\log n}$. 
    By \cref{lem:smallLevels}, we have that 
    $$
    \abs{L_{\ceil{\log n}}(G_\ell,\vec{S}_{\leq \ell})} 
    \quad \leq\quad 
    \frac{n}{2^{\ceil{\log n}}} \frac{\ell}{a}
    \quad  <\quad 
    1
    $$ 
    so $L_{\ceil{\log n}}(G_\ell,\vec{S}_{\leq \ell}) = \mt$. Therefore, there are fewer than $k \ceil{\log n }$ of the sets in $S$ intersecting $\hat{S}_\ell$. 
    Taking the subsequence not intersecting $\hat{S}_\ell$ and applying the induction hypothesis lets us conclude.
\end{proof}

To make use of the above, we will require a recent result of Chudnovsky, Codsi, and Kaneshiro~\cite{chudnovsky2026coarseLayeredSets}.
Recall that, for a positive integer $t$, we denote by $\mathcal{F}_t$ the class of $K_{t,t}$-induced-minor-free graphs.

\begin{proposition}
\label{r distance lemma}
    For all $t,r\in\nat$, there exists a positive integer $\mu=\mu(r,t)$ such that the following holds.         
    Let $C,s\in \nat$ such that $C\geq 2$, let $G\in \mathcal{F}_t$ with $\omega(G)<s$, and let $Y \subseteq V(G)$.
    Then there is a subset $X \subseteq V(G)$ with $|X| \leq (Cs)^\mu$ such that for every $v\in G\sm X$ we have that 
    $$
    |N^{r}_{G\sm X}[v]\cap Y|
    \ \ \leq\ \ 
    \frac{|Y|}{C}.
    $$
\end{proposition}

We denote the set $Z$ obtained by using \cref{r distance lemma} to a set $Y$ as $Z(G,Y,r,C)$.

\begin{lemma}\label{Lemma: result layered sets}
    Let $r,t\in \nat$. There exists a positive integer $\mu=\mu(r,t)$ such that the following holds
    Let $a,k,n,p,s$ be positive integers such that $n\ge 2$ and $a> pk\ceil{\log n}$ and let $\prop$ be a hereditary graph property. 
    Let $G$ be a $n$-vertex graph in a hereditary $(k,r,\prop)$-close subclass of $\mathcal{F}_t$ such that $\omega(G)<s$.
    There exists a set $T\subseteq V(G)$, with $|T|\leq(2aks)^{\mu}$ such that either $G$ is $T$-close to $\prop$ or there exist sets $S_1,\dots,S_p \subseteq V(G)$, pairwise disjoint and anticomplete, containing at most $k$ connected components, internally $(k,r)$-coverable and such that $G\sm T$ is $S_i$-close to $\prop$.
\end{lemma}

\begin{proof}
    Let $\mu = \mu'(2r+1,t)+3$ where $\mu'$ is defined as in \cref{r distance lemma}.
     Let $G_0=G$, $\hat{S}_0=\mt$ and $T_0 = Z(G,G,2r+1,2ka)$. Set \(\ell=pk\ceil{\log n}\). For every $0< j\leq \ell$, let $G_{j} =G_{j-1}\sm (T_{j-1}\cup \hat{S}_{j-1})$, $S_j$ be a $(k,r)$-coverable set in $G_j$  with cover $\hat{S}_j$ such that $G_j$ is $S_j$-close to $\prop$. By enlarging $S_j$ if necessary, we may assume that $S_j=N^r_{G_j}[\hat S_j]$. Let $S_j' = N_{G_j}^{2r+1}[\hat{S}_j]$. Let $$T_j = \bigcup_{i\leq j} Z\left(G_j,L_i(G_j,\vec{S}_{\leq j}'),2r+1,2ak\right).$$ Finally, let $T=\bigcup_{0\leq j\leq \ell} \left( T_j \cup \hat{S}_j \right)$. By \cref{r distance lemma}, we have that $|T| \leq a^2 (2aks)^{\mu'(2r+1,t)} + ka \leq (2aks)^{\mu}$.
     
     If $G$ is $T$-close to $\prop$ there is nothing to show. Therefore, we may assume that $G\sm T \notin \prop$.
     We check that \(\vec S'\), \(\vec{\hat S}\), and \(\vec G\) form a \((k,2r+1,a)\)-cleaning sequence. We prove this by induction on $\ell$. 
     If $\ell = 1$, we have that we have that $G_1 = G\sm T_0$ and so $G_1$ is $(ka,2r+1)$-clean with respect to $G$, and the other conditions to be a $(k,2r+1,a)$-cleaning sequence are clearly satisfied. We may assume that $\vec{S}_{\leq\ell-1}'$, $\vec{\hat{S}}_{\leq\ell-1}$ and $\vec{G}_{\leq\ell-1}$ forms a $(k,2r+1,a)$-cleaning sequence. Since $(k,2r+1,a)$-cleaning sequences are closed under taking prefixes, it suffices to verify that $G_\ell$ is $(2ka,2r+1)$-clean with respect to $\vec{S'}_{\leq \ell-1 }$. For every $m\in \nat$ and $v\in G_\ell$, we have that 
     \begin{align*}
         \abs{L_m\left(N_{G_\ell}^{2r+1}[v] ,\vec{S}_{\leq \ell-1 }'\right)} \quad
         &\leq\quad \frac{\abs{L_m(G_{\ell-1},\vec{S}_{\leq \ell-1 }')}}{2ka}\\
         &\leq\quad \frac{|G|}{2^m} \frac{\ell-1}{a} \frac{1}{2ka}\\
         &\leq\quad \frac{|G|}{2^m} \frac{1}{2ka}.
     \end{align*}
     Therefore,  $\vec{S}, \vec{S'},\vec{G}$ form a $(k,2r+1,a)$-cleaning sequence.


    We apply \cref{lemma: layered sets disjoint cores} to the cleaning sequence
$\vec S'$, $\vec{\hat S}$, and $\vec G$. Thus, there are indices
$i_1<\cdots<i_p$ such that $\hat S_{i_e}\cap S'_{i_f}=\emptyset$ for every
$e\neq f$.
For $e\in[p]$, let $R_e=S_{i_e}$. We claim that $R_1,\dots,R_p$ satisfy the
conclusion. Fix $e<f$. Since $\hat S_{i_f}\cap S'_{i_e}=\emptyset$ and
$S'_{i_e}=N^{2r+1}_{G_{i_e}}[\hat S_{i_e}]$, while
$\hat S_{i_f}\subseteq V(G_{i_f})\subseteq V(G_{i_e})$, we have
$\dist_{G_{i_e}}(\hat S_{i_e},\hat S_{i_f})>2r+1$. Since
$G_{i_f}\subseteq G_{i_e}$, if $R_e$ and $R_f$ were not disjoint and
anticomplete, then there would be a path in $G_{i_e}$ of length at most
$2r+1$ from $\hat S_{i_e}$ to $\hat S_{i_f}$, a contradiction. Thus the sets
$R_1,\dots,R_p$ are pairwise disjoint and anticomplete.

Moreover, each $R_e$ is internally $(k,r)$-coverable, with cover
$\hat S_{i_e}\subseteq R_e$, and has at most $k$ connected components. Finally,
since $G\sm T\subseteq G_{i_e}$ and $G_{i_e}\sm R_e\in\prop$, the heredity of
$\prop$ implies that $(G\sm T)\sm R_e\in\prop$. Thus $G\sm T$ is $R_e$-close to
$\prop$ for every $e\in[p]$, and the sets $R_1,\dots,R_p$ satisfy the conclusion.
     
\end{proof}

Next, we will specialize \cref{Lemma: result layered sets} to two specific hereditary properties.
We need the following theorem of Chudnovsky, Codsi, E S, and Lokshtanov~\cite{coarsetw}. Their statement is phrased in terms of the \defn{distance-$r$-independence number} rather than coverability, but a straightforward argument, given in the fourth paragraph of their preliminaries, shows that the two formulations are equivalent; we state the version below.

\begin{proposition}\label{prop:coarse_treewidth_our_class}
    Let $t$ be a positive integer and let $G\in\calc_t$ be an $n$-vertex graph. Then there exists a positive integer $d(t)$ and $\epsilon(t)\in(0,1]$, such that $G$ has a tree decomposition $(T,\chi)$ where for every $x\in V(T)$, $\chi(x)$ is $(2^{d(t) \log^{1-\epsilon(t)} n},8)$-coverable in $G$.
\end{proposition}

\begin{remark}
    For a positive integer $t$ and a graph $G\in\calc_t$, we use $\kball(t,n)$ to denote $2^{d(t) \log^{1-\epsilon(t)} n}$, when $t,n$ are clear from context, we simply write $\kball$. Similarly, we use $\dball$ to denote $d(t)$ and $\eball$ to denote $\epsilon(t)$.
\end{remark}


We note that for every weight function $w$, being $(w,\varepsilon)$-balanced is a hereditary property of labeled graphs (using $w$ as a labeling function on the vertices). 
A set $X \subseteq V(G)$ is a $(w,\varepsilon)$-balanced separator of $G$ if $G$ is $X$-close to be $(w,\varepsilon)$-balanced. We say that $G$ admits $(k,r)$-balanced separators if for every weight function, $G$ is $(k,r,((w,\varepsilon)\text{-balanced}))$-close. 
Similarly, a hereditary class of graphs is $(k,r,\varepsilon)$-separable if every graph in it is $(k,r,\varepsilon)$-separable.
We say that a hereditary class $\mathcal G$ is \defn{$(f(n),r,\varepsilon)$-separable} if every $n$-vertex graph $G\in\mathcal G$ admits a $(f(n),r)$-coverable $(w,\varepsilon)$-balanced separator for every weight function $w$ on $V(G)$.

\begin{corollary}\label{cor:balanced_separator_our_class}
Let $t$ be a positive integer. Then $\mathcal C_t$ is
$(4\kball(t,n),8,\frac14)$-separable.
\end{corollary}

\begin{proof}

Let \(G\in\calc_t\) be an \(n\)-vertex graph and let \(w\) be a weight function on \(V(G)\).
If $w(G)=0$, then the empty set is a $(w,\frac{1}{4})$-balanced separator, so we may
assume that $w(G)>0$. Let $(T,\chi)$ be a tree decomposition of $G$ as in
\Cref{prop:coarse_treewidth_our_class}. Applying \cref{prop: central bag} to
$(T,\chi)$ and $w$ gives a bag $S_1$ which is a $(w,\frac{1}{2})$-balanced separator
in $G$. In particular, every component of $G\sm S_1$ has weight at most
$1/2$.

Let $C_1,\ldots,C_m$ be the components of $G\sm S_1$ with weight more than
$1/4$. Since these components are pairwise disjoint, we have $m\leq 3$.
For each $i\in[m]$, define a weight function $w_i$ on $G[C_i]$ by
$w_i(v)=w(v)/w(C_i)$. Since $\mathcal C_t$ is hereditary, $G[C_i]\in\mathcal C_t$.
Applying \Cref{prop:coarse_treewidth_our_class} to $G[C_i]$ and \cref{prop: central bag} to $w_i$, we obtain a set $S_i'\subseteq C_i$ which is a
$(w_i,\frac12)$-balanced separator in $G[C_i]$ and is
$(\kball(t,|C_i|),8)$-coverable in $G[C_i]$. Since $|C_i|\leq n$, this implies
that $S_i'$ is $(\kball(t,n),8)$-coverable in $G$.

Let $S:=S_1\cup\bigcup_{i=1}^m S_i'$. Then $S$ is
$(4\kball(t,n),8)$-coverable in $G$, since $S_1$ and each $S_i'$ can be covered
by at most $\kball(t,n)$ balls of radius $8$, and $m\leq 3$.

It remains to show that $S$ is $(w,\frac14)$-balanced. Let $C$ be a component
of $G\sm S$. If $C$ is contained in a component of $G\sm S_1$ that is not among
$C_1,\ldots,C_m$, then $w(C)\leq 1/4$. Otherwise,
$C\subseteq C_i\sm S_i'$ for some $i\in[m]$. Since $S_i'$ is
$(w_i,\frac12)$-balanced in $G[C_i]$, we have $w_i(C)\leq 1/2$, and therefore
$w(C)\leq w(C_i)/2\leq 1/4$. Thus $S$ is a $(w,\frac14)$-balanced separator
in $G$.
\end{proof}

To make the results of this section more readily usable in the next sections, we introduce some definitions:
\begin{itemize}[label = --, leftmargin = 2em]
    \item A \defn{clean-set-packing} in $G$ is a pair $\bigl(D,\set{(S_i,C_i):i\in[\ell]}\bigr)$, where $D\subseteq V(G)$ and $S_i,C_i\subseteq V(G)$ for $i\in[\ell]$, satisfying the following, with $D':=D\cup\bigcup_{i=1}^\ell C_i$: for every distinct $i,j\in[\ell]$, the sets $S_i\cup C_i$ and $S_j\cup C_j$ are disjoint and anticomplete in $G$. 
    
    \item We call $\ell$ the \defn{length} of the clean-set-packing.
    Let $c,k$ be positive integers.
    A clean-set-packing is \defn{$k$-thick} if for every $i\in[\ell]$, the subgraph of $G-(D'\setminus C_i)$ induced by $S_i\cup C_i$ has at most $k$ connected components and \defn{$c$-costly} if $|D'| = c$.

    \item A clean-set-packing $(D,\set{(S_i,C_i):i\in[\ell]})$ in $G$, is a \defn{\balpack} if for every $i\in[\ell]$ the set $S_i$ is a $(Y',\frac{1}{2})$-balanced separator in $G'$, where $G':= G-D'$ and $Y' := Y\cap V(G')$.

\end{itemize}

\begin{lemma}\label{lem:Y_packing_new}
Let $t,\ell$ be positive integers. There exists a positive integer
$\mu=\mu(t)$ such that the following holds. Let $G\in\calc_t$ be an
$n$-vertex graph with $n\geq 2$ and $\omega(G)\leq \omega$, and let
$Y\subseteq V(G)$. Set
\[
K:=4\kball(t,n),\qquad
h':=\ell K(\ceil{\log n}+1),\qquad
\kappa:=\bigl(2h'K(\omega+1)\bigr)^\mu.
\]
Then either $G$ has a $(Y,\frac12)$-balanced separator of size at most
$\kappa$, or $G$ has a $(Y,\frac12)$-packing of length $\ell$ that is
$K$-thick and has cost at most $\kappa$.
\end{lemma}

\begin{proof}
If $Y=\emptyset$, then the empty set is a $(Y,\frac12)$-balanced separator, so
we may assume that $Y\neq\emptyset$. Define a weight function
$w:V(G)\to\mathbb R_{\geq 0}$ by setting $w(v)=1/|Y|$ for $v\in Y$, and
$w(v)=0$ otherwise. Let $\mathcal P$ be the hereditary property of labeled
induced subgraphs $H\subseteq G$ such that $H$ is $(w,\frac14)$-balanced.

We first check that every induced subgraph of $G$ is $(K,8,\mathcal P)$-close.
Let $H\subseteq G$ be an induced subgraph. Since $\calc_t$ is hereditary,
$H\in\calc_t$. By \cref{cor:balanced_separator_our_class}, applied to $H$ and
the restricted weight function $w|_{V(H)}$, there exists a
$(w|_{V(H)},\frac14)$-balanced separator $X$ in $H$ which is
$(4\kball(t,|H|),8)$-coverable in $H$. Since $|H|\leq n$, the set $X$ is
$(K,8)$-coverable in $G$, and $H\sm X\in\mathcal P$. Thus $H$ is
$(K,8,\mathcal P)$-close.

Since $\calc_t\subseteq\mathcal F_t$, we may apply
\cref{Lemma: result layered sets} with parameters $k=K$, $r=8$, $p=\ell$,
$s=\omega+1$, and $a = h'$. Since
$h'>\ell K\ceil{\log n}$, we obtain a set $T\subseteq V(G)$ with
$|T|\leq\kappa$ such that either $G\sm T\in\mathcal P$, or there exist sets
$R_1,\dots,R_\ell$ which are pairwise disjoint and anticomplete, each have at
most $K$ connected components and are internally $(K,8)$-coverable, and such
that $(G\sm T)\sm R_i\in\mathcal P$ for every $i\in[\ell]$.

If $G\sm T\in\mathcal P$, then $T$ is a $(Y,\frac14)$-balanced separator, and
hence a $(Y,\frac12)$-balanced separator. Thus we may assume that the second
case holds. Set $Y':=Y\cap V(G\sm T)$. If $|Y'|\leq |Y|/2$, then $T$ is a
$(Y,\frac12)$-balanced separator, so we may assume that $|Y'|>|Y|/2$.

For every $i\in[\ell]$, set $C_i:=R_i\cap T$, $S_i:=R_i\setminus T$, and let
$D:=T\setminus\bigcup_{i\in[\ell]}C_i$. Then
$D':=D\cup\bigcup_{i\in[\ell]}C_i=T$, and $S_i\cup C_i=R_i$ for every $i$.
As above, $(D,\{(S_i,C_i):i\in[\ell]\})$ is a $K$-thick clean-set-packing of
length $\ell$ and cost at most $\kappa$.

It remains to check the separator condition. Since
$(G\sm T)\sm R_i=(G\sm D')\sm S_i\in\mathcal P$, every component of
$(G\sm D')\sm S_i$ contains at most $|Y|/4$ vertices of $Y$. Since
$|Y'|>|Y|/2$, this is at most $|Y'|/2$. Hence $S_i$ is a
$(Y',\frac12)$-balanced separator in $G\sm D'$ for every $i\in[\ell]$.
Therefore the clean-set-packing is a $(Y,\frac12)$-packing.
\end{proof}

\begin{remark}
    For positive integers $t,q,\ell,n,\omega$, with $n\geq 2$, we use $\kcost(t,\ell,n,\omega)$ to denote the function $\kappa$ defined in \Cref{lem:Y_packing_new}. We also use $\mucostbal(t)$ to refer to the constant implied by \Cref{lem:Y_packing_new} and when $t$ is clear from context, we simply write $\mucostbal$.
\end{remark}

For fixed \(A,B\subseteq V(G)\), let \(\mathcal P_{A,B}\) be the hereditary
property of induced subgraphs \(H\subseteq G\) such that no connected component
of \(H\) contains both a vertex of \(A\) and a vertex of \(B\).

We say that a graph \(G\) is \defn{\((k,r,q)\)-slim-pair separable} if, for
every two sets \(A,B\subseteq V(G)\) which are \(q\)-slim in \(G\), there exists
a \((k,r)\)-coverable set \(X\subseteq V(G)\) such that \(G\) is \(X\)-close to
\(\mathcal P_{A,B}\). Equivalently, every \(q\)-slim pair \(A,B\) admits a
\((k,r)\)-coverable \(A\)--\(B\) separator.
A graph class \(\mathcal G\) is \defn{\((k,r,q)\)-slim-pair separable} if every
graph \(G\in\mathcal G\) is \((k,r,q)\)-slim-pair separable.

We now prove a similar packing lemma for $A$--$B$ separators. To do so, we will need \cref{cor:AB_separator_our_class}. 

\begin{corollary}\label{cor:AB_separator_our_class}
Let $t,q$ be positive integers. Then $\calc_t$ is
$(q\kball(t,n),8,q)$-slim-pair separable. Equivalently, if $G\in\calc_t$ is an
$n$-vertex graph and $A,B\subseteq V(G)$ are $q$-slim in $G$, then $G$ is
$(q\kball(t,n),8,\mathcal P_{A,B})$-close.
\end{corollary}

Its proof is a coarse analogue of a well-known one (see \cite{diestel2025graph}), but we include it here for completeness. We will need the following:

\begin{proposition}[Erd\H{o}s–Pósa property for subtrees of a tree \cite{gyarfas1970helly}]\label{lem:Helly_subtrees}
Let $T$ be a tree and let $\mathcal T$ be a family of nonempty subtrees of $T$.
For every positive integer $q$, either $\mathcal T$ contains $q$ pairwise
vertex-disjoint subtrees, or there exists a set $Z\subseteq V(T)$ with
$|Z|<q$ such that $Z$ intersects every subtree in $\mathcal T$.
\end{proposition}

\begin{proof}[Proof of \cref{cor:AB_separator_our_class}]
Let $G\in\calc_t$ be an $n$-vertex graph, and let $A,B\subseteq V(G)$ be
$q$-slim in $G$. Let $(T,\chi)$ be a tree decomposition of $G$ as in
\Cref{prop:coarse_treewidth_our_class}. Thus every bag is
$(\kball(t,n),8)$-coverable in $G$.

For every $A$--$B$ path $P$ in $G$, let
\(
T_P=\{x\in V(T):\chi(x)\cap V(P)\neq\emptyset\}.
\)
Then $T_P$ is a subtree of $T$. We claim that the family
$\mathcal T=\{T_P:P\text{ is an }A\text{--}B\text{ path in }G\}$ contains no
$q$ pairwise vertex-disjoint subtrees. Indeed, if
$T_{P_1},\dots,T_{P_q}$ are pairwise vertex-disjoint, then the paths
$P_1,\dots,P_q$  are pairwise vertex-disjoint and anticomplete, contradicting that $A$ and $B$ are $q$-slim. Indeed, a common vertex, or an edge between two of the paths, would be contained
in some bag, forcing the corresponding subtrees to intersect.  

By \Cref{lem:Helly_subtrees}, there exists a set $Z\subseteq V(T)$ with
$|Z|<q$ such that $Z$ intersects every subtree in $\mathcal T$. Let
\[
X=\bigcup_{z\in Z}\chi(z).
\]
Then $X$ intersects every $A$--$B$ path in $G$. Hence no connected component of
$G\sm X$ contains both a vertex of $A$ and a vertex of $B$, and so
$G\sm X\in\mathcal P_{A,B}$.

Finally, since each bag $\chi(z)$ is $(\kball(t,n),8)$-coverable in $G$ and
$|Z|<q$, the set $X$ is $(q\kball(t,n),8)$-coverable in $G$. Thus $G$ is
$(q\kball(t,n),8,\mathcal P_{A,B})$-close, as desired.
\end{proof}

Here again, we introduce some notation to make the following result more easily usable in the next sections.  A clean-set-packing $(D,\set{(S_i,C_i):i\in[\ell]})$ in $G$, is an \defn{\abpack} if for every $i\in[\ell]$ the set $S_i$ is an $A'$--$B'$ separator in $G'$, where $G':= G-D', A' := A\cap V(G')$, and $B':= B\cap V(G')$.

\begin{lemma}\label{lem:AB_packing_new}
Let $t,q,\ell$ be positive integers. There exists a positive integer
$\mu=\mu(t)$ such that the following holds. Let $G\in\calc_t$ be an
$n$-vertex graph with $n\geq 2$ and $\omega(G)\leq \omega$, and let
$A,B\subseteq V(G)$ be $q$-slim in $G$. Set
\[
\tau:=q\kball(t,n),\qquad
h:=\ell \tau(\ceil{\log n}+1),\qquad
c_{A,B}:=\bigl(2h\tau(\omega+1)\bigr)^\mu .
\]
Then either $G$ has an $A$--$B$ separator of size at most $c_{A,B}$, or $G$
has an \abpack of length $\ell$ that is $\tau$-thick and has cost at most
$c_{A,B}$.
\end{lemma}

\begin{proof}
Let $\mathcal P:=\mathcal P_{A,B}$. We first check that every induced subgraph
of $G$ is $(\tau,8,\mathcal P)$-close. Let $H\subseteq G$ be an induced subgraph.
Since $\calc_t$ is hereditary, $H\in\calc_t$. Moreover, $A\cap V(H)$ and
$B\cap V(H)$ are $q$-slim in $H$. By \cref{cor:AB_separator_our_class}, there
is an $(A\cap V(H))$--$(B\cap V(H))$ separator $X$ in $H$ which is
$(q\kball(t,|H|),8)$-coverable in $H$. Since $|H|\leq n$, the set $X$ is
$(\tau,8)$-coverable in $G$, and $H\sm X\in\mathcal P$. Thus $H$ is
$(\tau,8,\mathcal P)$-close.

Since $\calc_t\subseteq\mathcal F_t$, we may apply
\cref{Lemma: result layered sets} with parameters $k=\tau$, $r=8$, $p=\ell$,
$s=\omega+1$, and $a = h$. Since
$h>\ell \tau\ceil{\log n}$, we obtain a set $T\subseteq V(G)$ with
$|T|\leq c_{A,B}$ such that either $G\sm T\in\mathcal P$, or there exist sets
$R_1,\dots,R_\ell$ which are pairwise disjoint and anticomplete, each have at
most $\tau$ connected components and are internally $(\tau,8)$-coverable, and such
that $(G\sm T)\sm R_i\in\mathcal P$ for every $i\in[\ell]$.

In the first case, $T$ is an $A$--$B$ separator of size at most $c_{A,B}$. Thus
we may assume that the second case holds. For every $i\in[\ell]$, set
$C_i:=R_i\cap T$, $S_i:=R_i\setminus T$, and let
$D:=T\setminus\bigcup_{i\in[\ell]}C_i$. Then
$D':=D\cup\bigcup_{i\in[\ell]}C_i=T$, and $S_i\cup C_i=R_i$ for every $i$.
Since the sets $R_1,\dots,R_\ell$ are pairwise disjoint and anticomplete,
$(D,\{(S_i,C_i):i\in[\ell]\})$ is a clean-set-packing of length $\ell$.

It is $\tau$-thick because, for every $i$, the graph induced by
$S_i\cup C_i=R_i$ in $G\sm(D'\setminus C_i)$ has at most $\tau$ connected
components. Also, its cost is $|D'|=|T|\leq c_{A,B}$. Finally, since
$(G\sm T)\sm R_i=(G\sm D')\sm S_i\in\mathcal P_{A,B}$, the set $S_i$ is an
$A'\!$--$B'$ separator in $G\sm D'$, where
$A':=A\cap V(G\sm D')$ and $B':=B\cap V(G\sm D')$. Hence this is an
\abpack.
\end{proof}

\begin{remark}
    For positive integers $t,q,\ell,n,\omega$, with $n\geq 2$, we use $\gthick(q,t,n)$ 
    and $\ccost(t,q,\ell,n,\omega)$ to denote the functions $\tau$
    and $c_{A,B}$ defined in \Cref{lem:AB_packing_new}. We also use $\mucostab(t)$ to refer to the constant implied by \Cref{lem:AB_packing_new} and when $t$ is clear from context, we simply write $\mucostab$.
\end{remark}

\section{Graphs in $\calc_t$ admit Strong Barriers.}\label{sec:barrier}

In this section we show that every graph $G\in\calc_t$, together with an \abpack $\Pi$, admits a strong-barrier with the additional properties stated in the main theorem of this section, \cref{thm:strong_barrier}. 
This combinatorial object is central to the arguments of Section~\ref{sec:absep}.
Let $G$ be a graph, $A,B\subseteq V(G)$, and let $\Pi = \{D,\{(S_i,C_i):i\in[\ell]\}\}$ be an \abpack in $G$.
Set $D':=D\cup\bigcup_{i=1}^\ell C_i$, $G':=G-D'$, $A':=A\cap V(G')$, and $B':=B\cap V(G')$.
%
%
We make the following definitions, all relative to a partition $(L,Z,R)$ of $V(G')$ and
the \abpack $\Pi$.
\begin{itemize}[label = --, leftmargin = 2em]
    \item $(L,Z,R)$ is a \defn{strong-barrier} corresponding to $\Pi$ in $G$, if $A'\subseteq L$, $B'\subseteq R$, and every $L$--$R$ path $P$ in $G'$ has      a subpath that is again an $L$--$R$ path, with the additional property that its interior is contained in $Z$.
    \item For $i\in[\ell]$, the \defn{fragments} of $S_i\cup C_i$, written $\calf_i$, are the connected components of $G[S_i\cup C_i]$. 
        Furthermore, $\calf := \bigcup_{i=1}^\ell\calf_i$ is the set of \defn{fragments} of $\Pi$. 
        Since the sets $S_i\cup C_i$ for $i\in[\ell]$ pairwise do not touch, $\calf$ is exactly the set of connected components of $G\bigl[\bigcup_{i=1}^\ell(S_i\cup C_i)\bigr]$.
    \item For a positive integer $\gamma$, a strong-barrier is \defn{$\gamma$-broad} if every $L$--$R$ path in $G'$ intersects at least $\gamma$ distinct       fragments of $\calf$.
    \item The \defn{active components} of a strong-barrier, written $\cala$, are the connected components of $G'[Z]$ with a neighbor in both $L$ and $R$.
    \item $(L,Z,R)$ is \defn{$(\lambda_1,\lambda_2)$-light} if $\bigl|\bigcup_{C\in\cala}V(C)\bigr|\leq \lambda_1$ and $\bigl|\{F\in\calf : V(F)\cap            C\neq\emptyset,\, C\in\cala\}\bigr|\leq \lambda_2$.
\end{itemize}
%
We record the following theorem by Chudnovsky, Codsi, Fischer, and Lokshtanov~\cite{subpolytw}, after which we state the main theorem of this section. Recall that two vertices $u,v$ of a graph are called \defn{false-twins} if $N(u) = N(v)$.

\begin{proposition} \label{prop:bipartite_lemma}
    Corresponding to every positive integer $t$, there exists a positive integer $g = g(t)$ such that the following holds. Let $G$ be a bipartite $K_{t,t}$-induced-minor-free graph with bipartition $(A,B)$ such that no two vertices in $B$ are false-twins. Then $|E(G)| \leq g |A|$.
\end{proposition}

\begin{remark}
    For a positive integer $t$, we use $\gcost(t)$ to denote the constant implied by \Cref{prop:bipartite_lemma}.
\end{remark}

\begin{restatable}{theorem}{strongbarrier}\label{thm:strong_barrier}
    Let $t$ be a positive integer and let $G\in\calc_t$ be a graph on $n$ vertices. 
    Let $k,\gamma,\ell$ be positive integers with $\gamma\geq\tsquig$ and $\ell\geq 8\gamma$, let $A,B\subseteq V(G)$, and let $\Pi$ be a $k$-thick \abpack of length $\ell$ in $G$. 
    Set,
    \[
        \lambda_1(\gamma,n,\ell)\ :=\ 16\gamma n/\ell
        \qquad\text{and}\qquad
        \lambda_2(\gamma,t,k)\ :=\ 16\,\gamma\, \tsquig\,\gcost(\tsquig)\, k.
    \]
    Then there exists a strong-barrier corresponding to $\Pi$ in $G$, that is $(\lambda_1(\gamma,n,\ell),\,\lambda_2(\gamma,t,k))$-light and
    $\gamma$-broad.
\end{restatable}

In order to properly sketch a proof of this theorem, we need to fix some objects and make some definitions.
For the rest of this section let us fix a positive integer $t$ and a graph $G\in\calc_t$. 
Fix positive integers $k,\gamma,\ell$ with $\gamma\geq \tsquig$ and $\ell\geq 8\gamma$.
Let $\lambda_1 := \lambda_1(\gamma,n,\ell)$ and $\lambda_2 := \lambda_2(\gamma,t,k)$.
Also fix a pair of vertex subsets $A,B\subseteq V(G)$, and a $k$-thick \abpack $\Pi = (D,\set{(S_i,C_i):i\in[\ell]})$ of length $\ell$ in $G$. 
As before, let $D':=D\cup\bigcup_{i=1}^\ell C_i$, $G':=G-D'$, $A' := A\cap V(G')$, and $B' := B\cap V(G')$.
%
%
Now, for $v\in V(G')$, we define,
$$
d(v)\ \ :=\ \ \min\Bigl\{ \ell,\ \min\left\{\bigl|\set{i\in[\ell] : V(P)\cap S_i\neq\emptyset}\bigr| : P \text{ is an $A'$--$v$ path in } G'\right\}\Bigr\}.
$$ 
Let $r$ be a positive integer. 
We denote using $\calb_r$ the set $\set{v\in V(G') : d(v)\leq r}$, and write $\bcalb_r$ for $V(G')\setminus\calb_r$.
Define $\ell':= \floor{\frac{\ell}{4\gamma}}$ and for future purposes, observe that $\ell \geq 8\gamma$ gives,
$
\ell/\ell'
=\ell/\lfloor\frac{\ell}{4\gamma}\rfloor 
\leq  \ell/(\frac{\ell}{4\gamma}-1)
\leq  \ell/(\frac{\ell}{8\gamma})
=  8\gamma 
$
%
For $i\in[\ell']$, we define the \defn{$i$-th layer} to be $\call_i:=\calb_{2i\gamma}\setminus\calb_{(2i-1)\gamma}$.

\medskip

We make the following simple observations regarding our definitions:
Since each $S_i$ is an $A'$--$B'$ separator in $G'$, we have $d(v)=\ell$ for every $v\in B'$.
Also, since the sets $S_i$ are pairwise disjoint, every $v\in V(G')$ lies in at most one $S_i$. 
Hence $d(v)\leq 1$ for every $v\in A'$ and $d(v)\leq d(u)+1$ for every edge $uv$ of $G'$.
Finally, the layers $\call_i$ for $i\in[\ell']$ pairwise do not touch in $G'$: for $i<j$, every $u\in\call_i$ and $v\in\call_j$ satisfy $d(u)+\gamma \leq 2i\gamma +\gamma \leq (2j-1)\gamma < d(v)$.
So $d(v)>d(u)+1$, and hence $u,v$ must be distinct and non-adjacent. 
With these definitions in hand, we can outline the proof of \Cref{thm:strong_barrier}

\medskip
\noindent \textbf{Proof Overview:}  
We show in \Cref{lem:is_a_strong_barrier} that for every $i\in[\ell']$, the triple $(\calb_{(2i-1)\gamma},,\call_i,,\bcalb_{2i\gamma})$ is a $\gamma$-broad strong-barrier, so it suffices to find some $i\in[\ell']$ for which this barrier is $(\lambda_1,\lambda_2)$-light.
%
For the $\lambda_1$ part, a simple averaging argument shows that at least half of the layers $\call_i$ have active components spanning at most $\lambda_1$ vertices in total, since the layers are pairwise disjoint and together contain only $n$ vertices. 
For the $\lambda_2$ part, we build an auxiliary graph $H$ whose vertices correspond to the fragments of $\Pi$ and the active components of the $\lambda_1$-light layers, with an edge between two vertices exactly when their corresponding sets touch in $G$.
We show that $H$ is $K_{\tsquig,\tsquig}$-induced-minor-free, arguing that an induced $K_{\tsquig,\tsquig}$-minor model in $H$ would yield a \skt{\tsquig} in $G$, contradicting \Cref{lem:squigglyktt}. 
Applying \Cref{prop:bipartite_lemma} to $H$ then bounds its number of edges, and a further averaging argument over the $\lambda_1$-light layers produces one layer whose active components have low total degree in $H$, i.e. one that is also $\lambda_2$-light.
One caveat is that Proposition~\ref{prop:bipartite_lemma} also requires the active components to be free of false-twins, which need not hold in $H$ — we get around this using $\gamma$-broadness to show every false-twin class among them has size less than $\tsquig$, so that contracting each class to a representative loses only a bounded factor.

\medskip
Now we move on to the formal statements and proofs.

\begin{lemma}\label{lem:is_a_strong_barrier}
    For every $i\in[\ell']$, the triple $(\calb_{(2i-1)\gamma},\,\call_i,\,\bcalb_{2i\gamma})$ is a $\gamma$-broad strong-barrier.
\end{lemma}
\begin{proof}
    Fix $i\in[\ell']$. 
    We have $A'\subseteq\calb_1\subseteq\calb_{(2i-1)\gamma}$, since $d(v)\leq 1$ for every $v\in A'$ and $(2i-1)\gamma\geq 1$.
    Similarly $B'\subseteq \bcalb_{2\ell'\gamma}\subseteq \bcalb_{2i\gamma}$, since $d(v)=\ell$ for every $v\in B'$ and $2i\gamma < \ell$.
    For the second condition, let $P$ be a $\calb_{(2i-1)\gamma}$--$\bcalb_{2i\gamma}$ path in $G'$.
    Define $u$ to be the last vertex of $P$ that lies in $V(P)\cap \calb_{(2i-1)\gamma}$ and let $v$ be the first vertex of $P$ after $u$ that lies in $V(P)\cap \bcalb_{2i\gamma})$. 
    Observe that both are well defined since the first and last vertex of $P$ lies in $\calb_{(2i-1)\gamma}$ and $\bcalb_{2i\gamma}$ respectively.
    Let $P'$ be the $u$--$v$ subpath of $P$, and observe that the interior of $P'$ is contained in $\call_i$.
    Hence $(\calb_{(2i-1)\gamma},\,\call_i,\,\bcalb_{2i\gamma})$ is strong-barrier.
    
    Now to show that it is $\gamma$-broad, let $P$ be a $u$--$v$ path with $u\in \calb_{(2i-1)\gamma}$ and $v\in \bcalb_{2i\gamma}$.
    Let $P''$ be an $A'$--$u$ path witnessing $d(u)$, and let $Q$ be an $A'$--$v$ path within the $A'$--$v$ walk obtained by concatenating $P''$ and $P$.
    Observe that $Q$ intersects more than $2i\gamma$ distinct sets $S_j$ as $d(v) > 2i\gamma$, while $V(P'')$ intersects at most $(2i-1)\gamma$ as $d(u)\leq(2i-1)\gamma$, and $v$ lies in at most one.
    Thus the interior of $P$ intersects at least $\gamma$ distinct sets $S_j$. 
    Since the sets $S_j\cup C_j$ for $j\in[\ell]$ pairwise do not touch in $G$, each fragment lies in a single $S_j\cup C_j$.
    Consequently, we conclude that $P$ intersects at least $\gamma$ distinct fragments.
\end{proof}

For each $i\in[\ell']$, let $\cala_i$ denote the set of active components of $(\calb_{(2i-1)\gamma},\,\call_i,\,\bcalb_{2i\gamma})$, which is a strong barrier by \Cref{lem:is_a_strong_barrier}. Furthermore, for $I\subseteq [\ell']$, let $\cala_I$ denote $\bigcup_{i\in I}\cala_i$.

\begin{lemma}\label{lem:light_layer_exists}
    There exists $i\in[\ell']$ such that $(\calb_{(2i-1)\gamma},\,\call_i,\,\bcalb_{2i\gamma})$ is $(\lambda_1, \lambda_2)$-light.
\end{lemma}
\begin{proof}
    Let $I\subseteq[\ell']$ denote the set of $i\in[\ell']$ such that $\call_i$ satisfies $\bigl|\bigcup_{C\in\cala_i}V(C)\bigr|\leq \lambda_1$. 
    We first show that $|I|\geq\ell'/2$. 
    Suppose for contradiction that $|[\ell']\setminus I|>\ell'/2$. 
    The layers $\call_i$, $i\in[\ell']$, are pairwise disjoint subsets of $V(G')$ with $\bigcup_{C\in\cala_i}V(C)\subseteq\call_i$ for every $i\in[\ell']$.
    Hence,
    \[
        \sum_{i\in[\ell']}|\call_i| \quad
        \geq\quad \sum_{i\in[\ell']\setminus I} \sum_{C\in\cala_i}|V(C)| \quad
        >\quad |[\ell']\setminus I|\cdot\lambda_1 \quad
        \geq \quad |[\ell']\setminus I|\cdot\frac{2n}{\ell'} \quad
        >\quad \frac{\ell'}{2}\cdot\frac{2n}{\ell'} \quad
        =\quad n.
    \]

    Define $\cala_I := \bigcup_{i\in I}\cala_i$.
    Construct an auxiliary graph $H$ with vertex set $\cala_I\cup\calf$, joining two vertices of $H$ by an edge whenever the vertex subsets of $G$ they represent touch in $G$.

    \begin{claim}\label{clm:H_is_Ktt_free}
        $H$ does not contain $K_{\tsquig,\tsquig}$ as an induced minor.
    \end{claim}
    \begin{claimproof}
        Suppose for contradiction that $H$ contains $K_{\tsquig,\tsquig}$ as an induced minor, with sets $\{X_1,\dots,X_{\tsquig}\},\{Y_1,\dots,Y_{\tsquig}\}\subseteq\cala_I\cup\calf$ forming the induced minor model. 
        For $Z\subseteq\cala_I\cup\calf$, write $V_G(Z):=\bigcup_{v\in Z}V(v)\subseteq V(G)$, where $V(v)$ denotes the vertex subset of $G$ corresponding to $v$. 
        We show that $\cals:=\{V_G(X_1),\dots,V_G(X_{\tsquig}),V_G(Y_1),\dots,V_G(Y_{\tsquig})\}$ is a \skt{\tsquig} in $G$.\\[-18pt]
        
        \begin{itemize}
            \item Every $v\in\cala_I\cup\calf$ corresponds to a connected vertex subset of $G$, and by definition of $E(H)$, if $u,v\in V(H)$ are adjacent then $V(u)$ and $V(v)$ touch and consequently $V(u)\cup V(v)$ is connected in $G$. 
            Since $H[X_i]$ and $H[Y_i]$ are connected for every $i\in[\tsquig]$, applying this observation along a spanning tree of $H[X_i]$ and $H[Y_i]$ shows that $G[V_G(X_i)]$ and $G[V_G(Y_i)]$ are connected as well.\\[-18pt]
            
            \item Next, fix distinct $i,j\in[\tsquig]$. 
            If $V_G(X_i)$ and $V_G(X_j)$ touch, then there exist $u\in X_i$ and $v\in X_j$ such that $V(u)$ and $V(v)$ touch.
            But then $uv\in E(H)$, which contradicts that $X_i$ is anticomplete to $X_j$ in $H$. 
            Hence $V_G(X_i)$ and $V_G(X_j)$ do not touch. 
            The same argument applied to $V_G(Y_i)$ and $V_G(Y_j)$ shows that they do not touch.\\[-18pt]

            \item Finally, fix $i,j\in[\tsquig]$ and observe that there must exist $u\in X_i$ and $v\in Y_j$ with $uv\in E(H)$.
            Thus $V(u)$ and $V(v)$ touch and consequently $V_G(X_i)$ and $V_G(Y_j)$ touch.\\[-15pt]
        \end{itemize}

        \noindent Thus $\cals$ is a \skt{\tsquig} in $G$, contradicting \Cref{lem:squigglyktt}.
    \end{claimproof}

    \noindent For this proof, we say that a vertex subset forms a \defn{twin class} if it is a maximal set of pairwise false-twins.

    \begin{claim}\label{clm:H_is_bipartite_with_bounded_twins}
        $H$ is a bipartite graph with $\cala_I$ and $\calf$ forming the two sides of the bipartition, and every twin class that contains a vertex of $\cala_I$ is contained in $\cala_I$ and has cardinality less than $\tsquig$.
    \end{claim}
    \begin{claimproof}
        The elements of $\calf$ are the connected components of $G\bigl[\bigcup_{i=1}^\ell(S_i\cup C_i)\bigr]$, hence pairwise do not touch in $G$, so $\calf$ is independent in $H$.
        Likewise, for $i\in I$ the elements of $\cala_i$ pairwise do not touch in $G'$ by definition, and for distinct $i,j\in[\ell']$ the elements of $\cala_i$ and $\cala_j$ lie in $\call_i$ and $\call_j$ respectively and hence do not touch in $G'$. 
        In either case, since $G'$ is an induced subgraph of $G$, the elements do not touch in $G$ either.
        So $\cala_I=\bigcup_{i\in I}\cala_i$ is also independent in $H$, and thus $H$ is bipartite with $\cala_I$ and $\calf$ forming the two sides of the bipartition. 
    
        Now, observe that each $C\in\cala_I$ is connected in $G'$ and has neighbors $u,v$ in $\calb_{(2i-1)\gamma}$ and $\bcalb_{2i\gamma}$ respectively. 
        Consequently there is a $\calb_{(2i-1)\gamma}$--$\bcalb_{2i\gamma}$ path in $G'$ whose interior is contained in $C$. 
        By \Cref{lem:is_a_strong_barrier}, the barrier $(\calb_{(2i-1)\gamma},\call_i,\bcalb_{2i\gamma})$ is $\gamma$-broad, so this path, and hence $C$, intersects at least $\gamma\geq\tsquig$ distinct elements of $\calf$. 
        Consequently as $\calf$ is independent in $H$, any twin class that contains a vertex of $\cala_I$ must be contained in $\cala_I$.
        Furthermore, since $\cala_I$ and $\calf$ are independent in $H$, if some twin class $T\subseteq\cala_I$ of $H$ had $|T|\geq \tsquig$, then $T$ together with $\tsquig$ common neighbors of $T$ in $\calf$, would induce a $K_{\tsquig,\tsquig}$ in $H$, contradicting \Cref{clm:H_is_Ktt_free}.
    \end{claimproof}
    
    \begin{claim}\label{clm:H_few_edges}
        $|E(H)| < \tsquig\gcost(\tsquig)|\calf|$.
    \end{claim}
    \begin{claimproof}
        For each twin class $T\subseteq \cala_I$ of $H$, fix a representative $C_T\in T$, and let $\cala_I^*:=\{C_T : T\subseteq \cala_I\text{ is a twin class of }H\}$.
        Let $H^*:=H[\cala_I^*\cup\calf]$. 
        As an induced subgraph of $H$, $H^*$ is also bipartite with $\cala_I^*,\calf$ forming the two parts of the bipartition.
        Furthermore, it contains no $K_{\tsquig,\tsquig}$ induced minor by \Cref{clm:H_is_Ktt_free}.
        Since by definition no two vertices of $\cala_I^*$ are false-twins, we may apply \Cref{prop:bipartite_lemma} to $H^*$ with $A:=\calf$ and $B:=\cala_I^*$ and conclude that $|E(H^*)|\leq \gcost(\tsquig)|\calf|$.
        Finally, since every vertex of $\cala_I$ has the same neighborhood in $H$ as the representative of its twin class and since every twin class has cardinality less than $\tsquig$ by \Cref{clm:H_is_bipartite_with_bounded_twins}, we have,
        $$
            |E(H)| 
            \ \ =\ \ \sum_{C\in\cala_I}\deg_H(C) 
            \ \ =\ \ \sum_{C_T \in \cala_I^\star}|T|\cdot\deg_H(C_T) 
            \ \ <\ \ \tsquig\sum_{T}\deg_H(C_T) 
            \ \ \leq\ \ \tsquig\gcost(\tsquig)|\calf|.
        $$
    \end{claimproof}

    \noindent Now, since $|I|\geq\ell'/2$, we conclude that there exists $i\in I\subseteq[\ell']\subseteq[\ell]$, such that:
    \begin{align*}
        \abs{\{F\in\calf : V(F)\cap C\neq\emptyset,\, C\in\cala_i\}}\qquad
        &\leq\qquad \frac{1}{|I|}\sum_{i\in I}\abs{\{F\in\calf : V(F)\cap C\neq\emptyset,\, C\in\cala_i\}} \\
        &\leq\qquad \frac{1}{|I|}\sum_{C\in\cala_I}\deg_H(C)
                &&\hspace{-105pt}=\quad \frac{|E(H)|}{|I|} \\
        &<\qquad \frac{\tsquig\gcost(\tsquig)|\calf|}{\ell'/2}
                &&\hspace{-105pt}\leq\quad \frac{2\tsquig\gcost(\tsquig)\ell k}{\ell'} \\
        &\leq\qquad 16\,\tsquig\,\gcost(\tsquig)\,\gamma k
                &&\hspace{-105pt}=\quad \lambda_2 .
    \end{align*}
    Here, the second inequality holds since every $F\in\calf$ with $V(F)\cap C\neq\emptyset$ for some $C\in\cala_i$, touches $C$ and is hence adjacent to it in $H$.
    The strict inequality follows from \Cref{clm:H_few_edges} together with $|I|\geq\ell'/2$, the next inequality holds since the \abpack $(D,\set{(S_j,C_j):j\in[\ell]})$ is $k$-thick, so that each $\calf_j$ has at most $k$ elements and $|\calf|=\sum_{j=1}^\ell|\calf_j|\leq\ell k$.
    The final inequality uses $\ell/\ell'\leq 8\gamma$.
\end{proof}
Combining \Cref{lem:is_a_strong_barrier} and \Cref{lem:light_layer_exists} immediately implies the main theorem, which we restate below for convenience.
\strongbarrier*

\begin{remark}
    For positive integers $\gamma,t,n,\ell,k$, with $\gamma \geq \tsquig$ and $\ell > 2\gamma$, we use $\lamone(\gamma,n,\ell)$ and $\lamtwo(\gamma,t,k)$ to denote the functions $\lambda_1(\gamma,n,\ell)$ and $\lambda_2(\gamma,t,k)$ defined in \Cref{thm:strong_barrier}.
\end{remark}

\section{Separating Slim Subsets.}\label{sec:absep}

For this section, fix $\gamma := \max\{\pslim,\tsquig\}$.
Let \defn{$T^\omega_t(n)$} denote the maximum over all graphs $G\in\calc_t$ with $|V(G)\leq n|$ and $\omega(G) \leq \omega$, and $\qslim$-slim subsets $A,B\subseteq V(G)$, the minimum size of an $A$--$B$ separator in $G$.
Now we state the main theorem of this section and provide a brief overview of the proof right below it.

\begin{restatable}{theorem}{slimseparable}\label{thm:slim_implies_separable}
    For every positive integer $t$, there exists a positive integer $\beta(t)$ such that, for every positive integer $\omega$,
    \[
        T^\omega_t(n) \ \ \leq\ \ \omega^{\mucostab}\cdot 2^{\beta(t)\log^{1-\eball/2}n}.
    \]
\end{restatable}

\medskip
\noindent \textbf{Proof Overview:}  The proof bounds $T_t^\omega(n)$ by a recursion, which we establish as follows.
Given $G\in\calc_t$ and $\qslim$-slim sets $A,B\subseteq V(G)$, we invoke \Cref{lem:AB_packing_new}, which either produces a small $A$--$B$ separator directly, in which case we are done, or produces a thick \abpack $\Pi = (D,\{(S_i,C_i)\}_{i\in[\ell]})$ of the desired length $\ell$. 
In the latter case we apply \Cref{thm:strong_barrier} to $\Pi$, obtaining a strong-barrier $(L,Z,R)$ that is both $\gamma$-broad and $(\lamone,\lamtwo)$-light.
Let $Z'$ denote the set of vertices belonging to active components of the barrier.
For each $\qslim$-slim pair of fragments intersecting $Z'$, we take a minimum separator between them in $G'[Z']$, and we let $\Psi$ be the union of
$D\cup\bigcup_{i\in[\ell]}C_i$ with all of these.
In \Cref{clm:psi_separates_balls_in_layer}, we use $\gamma$-broadness and $(\pslim,\qslim)$-slimness to show that $\Psi$ is an $A$--$B$ separator in $G$.
Since each of the smaller separators is computed inside a smaller instance of the same problem on the vertex set $Z'$, this yields the recursive inequality of \Cref{lem:recursion}, bounding $T^\omega_t(n)$ in terms of $T^\omega_t(\lamone)$.

The rest of the section is devoted to solving this recursion. \Cref{lem:small_absep} unrolls the recursion of \Cref{lem:recursion} into a closed-form bound on $T^\omega_t(n)$, expressed in terms of the depth of the recursion $D_\ell(n)$ needed to reduce $n$ below the base case threshold. 
\Cref{lem:fixing ell} then fixes a specific choice of $\ell = \ell_t(n)$ and bounds $D_\ell(n)$ appropriately. 
\Cref{thm:slim_implies_separable} then follows by plugging the bounds on $\ccost$ and $\lamtwo$, together with the bound on $D_\ell(n)$ from \Cref{lem:fixing ell}, into \Cref{lem:small_absep}.

\medskip
Now we move on to the formal statements and proofs.
 
\begin{lemma}\label{lem:recursion}
    Let $t,\omega$ be positive integers, and let $q := \qslim$.
    Let $\ell_t(n)$ be a non-decreasing function on the positive integers with $\ell_t(n) = o(n)$ such that there exists a positive integer $n$ satisfying $\ell_t(n) > 256\gamma^2$.
    Let $n_0 = n_0(t)$ be the smallest such positive integer.
    Then $T^\omega_t(n)$ obeys the recursive inequality
    \[
        T^\omega_t(n) \ \ \leq\ \
        \begin{cases}
            n_0, & n < n_0, \\[4pt]
            \min\big\{ n,\; \ccost(t,q,\ell_t(n),n,\omega)
            + \lamtwo(\gamma,t,\gthick(q,t,n))^2\cdot T^\omega_t\bigl(\lamone(\gamma,n,\ell_t(n))\bigr)\big\},
            & \text{otherwise}.
        \end{cases}
    \]
\end{lemma}
\begin{proof}
    We prove this by induction on $n$.
    For the base case observe that if $n < n_0$, then the inequality holds trivially. 
    Now, fix a positive integer $n \geq n_0$ and assume that the statement holds for all strictly smaller positive integers.
    Let $G\in \calc_t$ and $A,B\subseteq V(G)$ be the graph and the pair of vertex subsets of $G$ that realize $T^\omega_t(n)$.

    Assume that $|V(G)| = n$, otherwise the size of the smallest $A$--$B$ separator in $G$ has size less than $T_t^\omega(n-1)$, which is at most $T_t^\omega(n)$ as $\ccost$, $\lamone$ and $\lamtwo$ are non-decreasing in $n$.
    Since it is clear that $T^\omega_t(n)\leq n$, we focus on the second term. Let $\ell = \ell_t(n)$.
    By \Cref{lem:AB_packing_new} applied to the tuple $(G,A,B,q,\ell)$, $G$ either has an $(A,B)$-separator of size $\ccost(t,q,\ell,n,\omega)$, or a $\gthick(q,t,n)$-thick \abpack $\Pi$ in $G$ of length $\ell$ and cost at most $\ccost(t,q,\ell,n,\omega)$.
    Assume the latter since in the former case, we are done.
    %
    Now, as $\ell \geq \ell_t(n_0) \geq 8\gamma$ and as $\gamma \geq \tsquig$, we may apply \Cref{thm:strong_barrier} 
    to obtain a strong-barrier $(L,Z,R)$ corresponding to $\Pi$ which is $\gamma$-broad and $(\lamone(\gamma,n,\ell),\, \lamtwo(\gamma,t,\gthick(q,t,n)))$-light.
    Let $Z' := \bigcup_{C\in\cala}V(C)$, where $\cala$ denotes the active components of $(L,Z,R)$.
    Let $\calf$ denote the fragments of $\Pi$ and let $\calf' := \{F\in\calf : V(F)\cap Z'\neq\emptyset\}$.
    Consider the graph $G'[Z']$, and denote $V(F)\cap Z'$ using $F'$ for every $F\in\calf'$.
    Let,
    $$
    \hat{\calf}\ \ :=\ \ \{\{F_{i},F_{j}\}~:~F_{i},F_{j}\in\calf',\, F'_{i},F'_{j} \text{ is a $q$-slim pair in }G'[Z']\}.
    $$
    Let $\Psi := \bigcup_{\{F_i,F_j\}\in\hat{\calf}}\Psi_{ij}$, where, for each $\{F_i,F_j\}\in\hat{\calf}$, $\Psi_{ij}$ denotes the smallest $F'_i$--$F'_j$ separator in $G'[Z']$.

    \begin{claim}\label{clm:psi_separates_balls_in_layer}
        $D'\cup \Psi$ is an $L$--$R$ separator in $G$.
    \end{claim}
   \begin{claimproof}
        To prove the claim, we need only show that $\Psi$ is an $L$--$R$ separator in $G'$.
        Suppose not; then $G'-\Psi$ contains an $L$--$R$ path $P$.
        By definition of active components, there must exist a subpath $P'$ of $P$, with an endpoint each in $L$ and $R$, whose interior is contained in $Z'$.
        Then since $(L,Z,R)$ is $\gamma$-broad, the interior of $P'$ intersects at least $\pslim\leq\gamma$ distinct fragments $F_1,\dots,F_{\pslim}\in\calf'$ of $\Pi$.
        Let $\tilde{G}:=G\bigl[Z'\cup\bigcup_{i=1}^{\pslim}V(F_i)\bigr]$.
        Since $\calc_t$ is hereditary, $\tilde{G}\in\calc_t$.
        Furthermore, the sets in $\{V(F_1),\dots,V(F_{\pslim})\}$ pairwise do not touch, and are connected in $\tilde{G}$.
        Since $\tilde{G}$ is strongly-$(\pslim,\qslim)$-slim by \Cref{lem:pqslim} and since $q = \qslim$, there exists a pair $V(F_{i}),V(F_{j})$ that is $q$-slim in $\tilde{G}$.
        Since $G'[Z']$ is an induced subgraph of $\tilde{G}$, each $F'_{i}$--$F'_{j}$ path in $G'[Z']$ is also a $V(F_i)$--$V(F_j)$ path in $\tilde{G}$, so the pair $F'_{i}, F'_{j}$ is $q$-slim in $G'[Z']$.
        But then $\Psi$ separates $F'_{i}$ from $F'_{j}$ in $G'[Z']$, contradicting that the interior of $P'$ is a path in $G'-\Psi$ that intersects both $F'_{i}$ and $F'_{j}$.
    \end{claimproof}

    \noindent Now, from \Cref{clm:psi_separates_balls_in_layer} it follows that,
    \begin{align*}
        T_t^\omega(n)\quad
        &\leq\quad |D'| + |\Psi| \\
        &\leq\quad \ccost(t,q,\ell,n,\omega)
            \ \ +\ \ \sum_{\{F_i,F_j\}\in\hat{\calf}} |\Psi_{ij}| \\
        &\leq\quad \ccost(t,q,\ell,n,\omega)
            \ \ +\ \ \lamtwo(\gamma,t,\gthick(q,t,n))^2\cdot T^\omega_t\bigl(\lamone(\gamma,n,\ell)\bigr).
    \end{align*}
    The last inequality uses that $(L,Z,R)$ is $\bigl(\lamone(\gamma,n,\ell),\lamtwo(\gamma,t,\gthick(q,t,n))\bigr)$-light, which gives two bounds. 
    First, $|\hat{\calf}|\leq\lamtwo(\gamma,t,\gthick(q,t,n))^2$, since $\hat{\calf}\subseteq\calf'\times\calf'$. 
    Second, for every $\{F_i,F_j\}\in\hat{\calf}$, $|\Psi_{ij}|\leq T^\omega_t\bigl(\lamone(\gamma,n,\ell)\bigr)$, since $G'[Z']\in\calc_t$ with $|Z'|\leq\lamone(\gamma,n,\ell)$ and $\omega(G'[Z'])\leq\omega$, and $F'_i,F'_j$ form a $q$-slim pair.
\end{proof}

We remark that the recursion in \Cref{lem:recursion} is well defined, since $\lamone(\gamma,n,\ell_t(n)) < n$ whenever $\ell > 8\gamma$. 
Now, let $t$ be a positive integer, and let $\ell_t(n)$ be a non-decreasing function on the positive integers with $\ell_t(n) = o(n)$ such that there exists a positive integer $n$ satisfying $\ell_t(n) > 256\gamma^2$.
Let $n_0 = n_0(t)$ be the smallest such positive integer.
We define a function $D_\ell$ recursively as follows:
\[
  D_\ell(n)=
  \begin{cases}
    0 & \text{if } n < n_0,\\
    D_\ell(\lamone(\gamma, n, \ell_t(n))) + 1 & \text{otherwise}.
  \end{cases}
\]
Note that this is well defined, since $\lamone(\gamma,n,\ell_t(n)) < n$ whenever $\ell > 8\gamma$. 

\begin{lemma}\label{lem:small_absep}
    Let $t,\omega$ be positive integers, and let $q := \qslim$.
    Let $\ell = \ell_t(n)$ be a non-decreasing function on the positive integers with $\ell_t(n) = o(n)$ such that there exists a positive integer $n$ satisfying $\ell_t(n) > 256\gamma^2$.
    Let $n_0 = n_0(t)$ be the smallest such positive integer.
    Then $T^\omega_t(n)$ obeys the inequality 
    $$
    T^\omega_t(n)\ \ \leq\ \ (n_0+1)\cdot \ccost(t,q,\ell_t(n),n,\omega)\cdot\lamtwo(\gamma,t,\gthick(q,t,n))^{2D_\ell(n)}
    $$ \
\end{lemma}
\begin{proof}
    For ease of exposition, we write $c(n) := \ccost(t,q,\ell,n,\omega)$, $\lambda_1(n) := \lamone(\gamma,n,\ell)$, and $\lambda_2(n) := \lamtwo(\gamma,t,\gthick(q,t,n))$.
    To prove the lemma, for each positive integer $N$ we define a function on the positive integers by
    \[
        H_N(n) \ \ =\ \
        \begin{cases}
            n_0, & \text{if } n \leq n_0, \\[4pt]
            c(N) + \lambda_2(N)^2\cdot H_N\bigl(\lambda_1(n)\bigr), & \text{otherwise}.
        \end{cases}
    \]
    Since the functions $c$ and $\lambda_2$ are non-decreasing in $n$, $T_t^\omega(n) \leq H_N(n)$ for every positive integer $N$ and every $n \in [N]$. 
    Hence, to prove the lemma, it suffices to upper bound $H_N(N)$.
    To this end, we claim that
    \[
        H_N(n) \ \ \leq\ \ (n_0+1)\cdot c(N)\cdot\lambda_2(N)^{2D_\ell(n)}
            \;-\; \frac{c(N)}{\lambda_2(N)^2 - 1}
    \]
    for all positive integers $n,N$. 
    The right-hand side is well defined since $\lambda_2(N) > 1$ for every $N$. 
    We fix an arbitrary positive integer $N$ and prove the claim for that $N$ by induction on $n$.
    For the base case, observe that when $n < n_0$, since $\lambda_2(N) \geq 2$ and $c(N)\geq 1$, we have,
    \begin{align*}
        (n_0+1)\cdot c(N)\cdot\lambda_2(N)^{2D_\ell(n)} - \frac{c(N)}{\lambda_2(N)^2 - 1} \quad 
        &\geq\quad (n_0+1)\cdot c(N) - c(N) \\
        &\geq\quad n_0 \quad =\quad H_N(n).
    \end{align*}
    For the inductive step, fix a positive integer $n \geq n_0$ and assume that the statement holds for all positive integers smaller than $n$.
    We have,
    
    \begin{align*}
        H_N(n)\ \
        &= \ \ c(N) + \lambda_2(N)^2\cdot H_N\bigl(\lambda_1(n)\bigr) \\
        &\leq \ \ c(N) + \lambda_2(N)^2\cdot
            \Biggl((n_0+1)\cdot c(N)\cdot \lambda_2(N)^{2D_\ell(\lambda_1(n))} \;-\; \frac{c(N)}{\lambda_2(N)^2 - 1} \Biggr) \\
        &\leq \ \ c(N) + (n_0+1)\cdot c(N)\cdot \lambda_2(N)^{2(D_\ell(\lambda_1(n)) + 1)}
            \;-\; c(N)\cdot \frac{\lambda_2(N)^2}{\lambda_2(N)^2 - 1} \\
        &\leq \ \ (n_0+1)\cdot c(N)\cdot \lambda_2(N)^{2D_\ell(n)}
            \;-\; c(N)\left(\frac{\lambda_2(N)^2}{\lambda_2(N)^2 - 1} - 1 \right) \\
        &\leq \ \ (n_0+1)\cdot c(N)\cdot \lambda_2(N)^{2D_\ell(n)}
            \;-\; \frac{c(N)}{\lambda_2(N)^2 - 1}.
    \end{align*}
    This proves our claim and consequently the lemma.
\end{proof}

\begin{lemma}\label{lem:fixing ell}
    Let $t$ be a positive integer, and let $\ell = \ell_t(n) := 2^{\dball\log^{1-\frac{\eball}{2}}n}$. Then $D_\ell(n)$ is well-defined and at most $\tfrac{4}{\dball\eball}\log^{\eball/2} n$.
\end{lemma}
\begin{proof}
    Since $O(\log^{1-\eball/2}n) \subseteq o(\log n)$, we have $\ell_t(n) = o(n)$, and since $\ell_t$ is a monotonically increasing function, $\ell_t(n) > 256\gamma^2$ for some positive integer $n$.
    Consequently $D_\ell(n)$ is well-defined.
    Let $n_0$ be the smallest integer $n$ such that $\ell_t(n) > 256\gamma^2$. 
    For each positive integer $n$, write $\lambda_1(n) := \lamone(\gamma,n,\ell_t(n)) = 16\gamma n/\ell_t(n)$.
    By definition of $n_0$ and monotonicity of $\ell_t$, for every $n\geq n_0$ we have $\log(\ell_t(n)) \geq \log(\ell_t(n_0)) > 2\log(16\gamma)$, and hence
    \[
        \log n - \log\lambda_1(n) \ =\ \log n\; -\; \left(\log(16\gamma) + \log n - \log(\ell_t(n))\right)\ =\ \log(\ell_t(n)) - \log(16\gamma) \ \geq\ \tfrac{1}{2}\log(\ell_t(n)).
    \]

    We show by strong induction on $n$ that $D_\ell(n)\leq \tfrac{4}{\dball\eball}\log^{\eball/2} n$ for every positive integer $n$.
    For the base case, observe that when $n < n_0$, we have $D_\ell(n)=0\leq \tfrac{4}{\dball\eball}\log^{\eball/2} n$, as required.
    For the inductive step, fix a positive integer $n \geq n_0$ and assume that the statement holds for all positive integers smaller than $n$.
    Now, we get the following inequality,
    \begin{align*}
        \log^{\eball/2}n - \log^{\eball/2}\lambda_1(n)\quad
        &\geq\quad \tfrac{\eball}{2}\log^{\eball/2-1}n\cdot\bigl(\log n-\log\lambda_1(n)\bigr) \\
        &\geq\quad \tfrac{\eball}{2}\log^{\eball/2-1}n\cdot\tfrac{1}{2}\log(\ell_t(n)) \\
        &\geq\quad \tfrac{\eball}{2}\log^{\eball/2-1}n\cdot\tfrac{1}{2}\dball\log^{1-\eball/2}n \\
        &=\quad \tfrac{\dball\eball}{4}.
    \end{align*}
    Here the first inequality applies the mean value theorem to the function $x\mapsto x^{\eball/2}$ on $[\log\lambda_1(n),\log n]$: since the derivative $\frac{\eball}{2}x^{\eball/2-1}$ is decreasing in $x$ as $\eball/2-1<0$, it is minimized at $x=\log n$, giving the lower bound $\frac{\eball}{2}\log^{\eball/2-1}n\cdot(\log n-\log\lambda_1(n))$. The second uses $n \geq n_0$ combined with our earlier observation, and the last one substitutes $\log(\ell_t(n))=\dball\log^{1-\eball/2}n$.
    Combining this with the definition of $D_\ell$ and the inductive assumption, yields the following. 
    \begin{align*}
        D_\ell(n) \quad &=\quad D_\ell(\lambda_1(n)) + 1 \\
        &\leq\quad \tfrac{4}{\dball\eball}\log^{\eball/2}(\lambda_1(n)) + 1 \\
        &\leq\quad \tfrac{4}{\dball\eball}\Bigl(\log^{\eball/2} n - \tfrac{\dball\eball}{4}\Bigr) + 1 \\
        &=\quad \tfrac{4}{\dball\eball}\log^{\eball/2} n.
    \end{align*}
    This completes the induction and thereby the proof of the lemma.
\end{proof}

Now we are ready to prove the main theorem of this section which we restate here for convenience.

\slimseparable*
\begin{proof}
    Let us assume without loss of generality that $n \geq 2$, since otherwise the theorem holds trivially.
    Fix $t$, let $\omega$ be a positive integer, set $q:=\qslim$, and let $\ell=\ell_t(n):=2^{\dball\log^{1-\eball/2}n}$.
    Let $n_0$ be the smallest integer $n$ such that $\ell_t(n) > 256\gamma^2$, which exists since $\ell_t$ is a monotonically increasing function on $n$.
    By \Cref{lem:small_absep},
    \[
        T^\omega_t(n) \ \ \leq\ \ (n_0+1)\cdot\ccost(t,q,\ell,n,\omega)\cdot\lamtwo(\gamma,t,\gthick(q,t,n))^{2D_\ell(n)}. \tag{$\dagger$}
    \]

    \noindent We begin by upper bounding $\ccost(t,q,\ell,n,\omega)$:
    \begin{align*}
        \ccost(t,q,\ell,n,\omega)\quad 
        &=\quad \Bigl(2h\cdot(\omega+1)\cdot\gthick(q,t,n)\Bigr)^\mu  \\
        &=\quad \Bigl(2\cdot\bigl(\ell\, q\kball(\ceil{\log n}+1)\bigr)\cdot q\kball\cdot (\omega+1)\Bigr)^\mu\\
        &\leq\quad c_1\cdot \omega^{\mucostab} \cdot \left({\kball}^{2}\cdot \ell\cdot (\ceil{\log n}+1)\right)^{\mucostab} \\
        &\leq\quad \omega^{\mucostab} \cdot 2^{\log c_1\, +\, \mucostab(2\dball\log^{1-\eball}n\, +\, \dball\log^{1-\eball/2}n\, +\, c_2\log\log n)}\\
        &\leq\quad \omega^{\mucostab} \cdot 2^{c_3\,\log^{1-\eball/2}n}\tag{1}
    \end{align*}
    Here $c_1,c_2$ and $c_3$ are suitable constants that depend only on $t$.
    The last inequality uses the fact that $O(\log\log n)\subseteq O(\log^{1-\eball}n)\subseteq o(\log^{1-\eball/2}n)$ and that $n \geq 2$.
    The rest of the transitions only involve substituting definitions or performing simplifications.
    Next, since $\gamma,\tsquig,\gcost(\tsquig)$ and $q$ depend only on $t$ and since $D_\ell(n)\leq\tfrac{4}{\dball\eball}\log^{\eball/2}n$, we have the following bound:
    \begin{align*}
        \lamtwo(\gamma,t,\gthick(q,t,n))^{2D_\ell(n)}
        \quad &=\quad \bigl(16\gamma\,\tsquig\,\gcost(\tsquig)\cdot q\kball\bigr)^{2D_\ell(n)} \\
        \quad &\leq\quad \bigl(c_4\,2^{\dball\log^{1-\eball}n}\bigr)^{\tfrac{8}{\dball\eball}\log^{\eball/2}n} \\
        \quad &\leq\quad 2^{c_5\,\log^{1-\eball/2}n}. \tag{2}
    \end{align*}
    Here $c_4$ and $c_5$ are suitable constants that depend only on $t$.
    Combining (1) and (2) via $(\dagger)$, gives
    \begin{align*}
        T^\omega_t(n)
        \quad &\leq\quad (n_0+1)\cdot \omega^{\mucostab}\cdot 2^{c_3\,\log^{1-\eball/2}n}\cdot 2^{c_5\,\log^{1-\eball/2}n} \\
        \quad &\leq\quad \omega^{\mucostab}\cdot 2^{\log (n_0+1)\, +\, (c_3+c_5)\log^{1-\eball/2}n} \\
        \quad &=\quad \omega^{\mucostab}\cdot 2^{\beta(t)\log^{1-\eball/2}n}.
    \end{align*}
    Here $\beta(t)$ is a suitably chosen constant that depends only on $t$. 
    This completes the proof of the theorem.
\end{proof}

\section{Proof of the Main Theorem.}\label{sec:mainthm}

In this section we complete the proof of Theorem~\ref{thm:atw_bound}, the main technical result of the paper. 
We first prove \Cref{lem:half_balanced_sep}, which claims the existence of a $(Y,\tfrac12)$-balanced separator for graphs in $\calc_t$, whose size depends on $\omega(G)$, via an argument similar to that of \Cref{lem:recursion}. 
We then use the results of Chudnovsky, E S, and Lokshtanov~\cite{alphacontainers} to complete the proof of the theorem, bounding the tree-independence number of $\calc_t$.

\begin{lemma}\label{lem:half_balanced_sep}
    For every positive integer $t$, there exists positive integers $\nu(t),\mu(t)$ such that, for every graph $G\in\calc_t$ on $n$ vertices and every vertex subset $Y\subseteq V(G)$, $G$ has a $(Y,\frac{1}{2})$-balanced separator of size at most $\omega(G)^{\mu(t)}\cdot2^{\nu(t)\log^{1-\eball/2}n}$.
\end{lemma}
\begin{proof}
    Let us assume without loss of generality that $n \geq 2$, since otherwise the theorem holds trivially.
    Write $p:=\pslim$, $q:=\qslim$, and $\omega:=\omega(G)$.
    By \Cref{lem:Y_packing_new} applied to $(G,Y)$ with $\ell = p$, $G$ either has a $Y$-balanced-separator of size $\kcost(t,p,n,\omega)$, or a \balpack $(D,\set{(S_i,C_i):i\in[p]})$ of length $p$ that is $4\kball$-thick and $\kcost(t,p,n,\omega)$-costly.
    Assume the latter since in the former case, we are done via the upper bound on $\kcost(t,p,n,\omega)$ that we prove towards the end of this lemma.
    Set $D':=D\cup\bigcup_{i=1}^p C_i$ and $G':=G-D'$, so that each $S_i$ is a $(Y,\frac{1}{2})$-balanced separator in $G'$.
    For every $i\in[p]$, let $\calf_i$ denote the set of connected components of $G[S_i\cup C_i]$, which, with slight abuse of notations, we call \emph{fragments}.
    Since $(D,\set{(S_i,C_i):i\in[p]})$ is $4\kball$-thick, we have $|\calf_i|\leq 4\kball$ for every $i\in [p]$.
    Let $\calf := \bigcup_{i\in[p]}\calf_i$ and let $F' := F\cap V(G')$ for every $F\in\calf$.
    Let,
    $$
    \hat{\calf}\quad :=\quad \{\{F_{i},F_{j}\}~:~F_{i},F_{j}\in\calf, F'_i,F'_j \text{ is a $q$-slim pair in } G'\}.
    $$
    For every $\{F_{i},F_{j}\}\in\hat{\calf}$, let $\Psi_{ij}$ denote the smallest $F'_{i}$--$F'_{j}$ separator in $G'$, which by \Cref{thm:slim_implies_separable} has size at most $T^\omega_t(n)$.
    Finally, let $\Psi := \bigcup_{\{F_i,F_j\}\in\hat{\calf}}\Psi_{ij}$.

    \begin{claim}\label{clm:half_balanced_after_psi}
        $D'\cup\Psi$ is a $(Y,\frac{1}{2})$-balanced separator in $G$.
    \end{claim}
    \begin{claimproof}
        To prove the claim, we need only show that $\Psi$ is a $(Y,\frac{1}{2})$-balanced separator in $G'$.
        Suppose not; then $G'-\Psi$ contains a component $K$ with $|Y\cap V(K)|>\frac{1}{2}|Y|$.
        Fix $i\in[p]$. Since $S_i$ is a $(Y,\frac{1}{2})$-balanced separator in $G'$, every component of $G'-S_i$ contains at most $\frac{1}{2}|Y|$ vertices of $Y$.
        Consequently $V(K)$ does not induce a connected subgraph of $G'-S_i$, which implies $V(K)\cap S_i\neq\emptyset$.
        Let $F_i\in\calf_i$ be such that $K\cap F'_i$ is non-empty.
        
        Define $\tilde{G} := G\bigl[V(G')\cup\bigcup_{i=1}^{p}V(F_i)\bigr]$.
        Since $\calc_t$ is hereditary, $\tilde{G}\in\calc_t$.
        Furthermore, the sets in $\{V(F_1),\dots,V(F_{p})\}$ pairwise do not touch, and are connected in $\tilde{G}$.
        Since $\tilde{G}$ is strongly-$(\pslim,\qslim)$-slim by \Cref{lem:pqslim} and since $q = \qslim$ and $p = \pslim$, there exists a pair $V(F_{i}),V(F_{j})$ that is $q$-slim in $\tilde{G}$.
        Since $G'$ is an induced subgraph of $\tilde{G}$, every $F'_{i}$--$F'_{j}$ path in $G'$ is also a $V(F_i)$--$V(F_j)$ path in $\tilde{G}$, so the pair $F'_{i}, F'_{j}$ is $q$-slim in $G'$.
        But then $\Psi$ separates $F'_{i}$ from $F'_{j}$ in $G'$, contradicting that $K$ is a connected component of $G'-\Psi$ that intersects both $F'_{i}$ and $F'_{j}$.
    \end{claimproof}

    Now, from \Cref{clm:half_balanced_after_psi} it follows that, $D'\cup\Psi$ is a $(Y,\frac{1}{2})$-balanced separator in $G$ of size at most $\kcost(t,p,n,\omega) + (4\kball p)^2\cdot T^\omega_t(n)$, since there are at most $\sum_{i=1}^p|\calf_i|\leq 4\kball p$ fragments in total, each contributing a separator of size at most $T^\omega_t(n)$.
    It remains to bound $\kcost(t,p,n,\omega) + (4\kball p)^2\cdot T^\omega_t(n)$.
    We begin with the following,
    \begin{align*}
        \kcost(t,p,n,\omega)\quad 
        &=\quad \Bigl(2h'\cdot(\omega+1)\cdot 4\kball\Bigr)^{\mucostbal}  \\
        &=\quad \Bigl(2\cdot\bigl(p\, 4\kball(\ceil{\log n}+1)\bigr)\cdot 4\kball\cdot (\omega+1)\Bigr)^{\mucostbal}\\
        &\leq\quad c_1\cdot \omega^{\mucostbal} \cdot \left({\kball}^{2}\cdot (\ceil{\log n}+1)\right)^{\mucostbal} \\
        &\leq\quad \omega^{\mucostbal} \cdot 2^{\log c_1\, +\, \mucostbal(2\dball\log^{1-\eball}n\, +\, c_2\log\log n)}\\
        &\leq\quad \omega^{\mucostbal} \cdot 2^{c_3\,\log^{1-\eball}n}\tag{1}
    \end{align*}
    Here $c_1,c_2,c_3$ are suitable constants that depend only on $t$.
    The last inequality uses the fact that $O(\log\log n)\subseteq o(\log^{1-\eball}n)$.
    The rest of the transitions only involve substituting definitions or performing simplifications.
    Combining this with the upper bound on $T^\omega_t(n)$ from \Cref{thm:slim_implies_separable} yields,
    \begin{align*}
        \kcost(t,p,n,\omega)\; +\; (4\kball p)^2\cdot T^\omega_t(n)
        \ \ &\leq\ \ \omega^{\mucostbal}\cdot 2^{c_3\log^{1-\eball}n}\; +\; (4\kball p)^2\cdot \omega^{\mucostab}\cdot 2^{\beta(t)\log^{1-\eball/2}n} \\
        \ \ &\leq\ \ \omega^{\mucostbal}\cdot 2^{c_3\log^{1-\eball}n}\; +\; c_4\cdot 2^{2\dball\log^{1-\eball}n}\cdot \omega^{\mucostab}\cdot 2^{\beta(t)\log^{1-\eball/2}n} \\
        \ \ &\leq\ \ \omega^{\mu(t)}\cdot 2^{c_3\log^{1-\eball}n\; +\; \log c_4\; +\; 2\dball\log^{1-\eball}n\; +\; \beta(t)\log^{1-\eball/2}n} \\
        \ \ &\leq\ \ \omega^{\mu(t)}\cdot 2^{\nu(t)\log^{1-\eball/2}n},
    \end{align*}
    Here $c_4$, $\mu(t)$ and $\nu(t)$ are suitable constants that depend only on $t$, which completes the proof of the lemma.
\end{proof}





To prove the main theorem, we need the following two results by Chudnovsky, E S, and Lokshtanov~\cite{alphacontainers}.
To this end, we make the following definitions.
A family $\mathcal{F}$ of vertex subsets is a \defn{$(b,a)$-container family} in $G$, if $\alpha(F)\leq a$ for every element $F$ in $\mathcal{F}$, and for every vertex subset $Z$ satisfying $\alpha(Z) \leq b$, there exists some element $F$ of $\mathcal{F}$ such that $Z\subseteq F$.
The \defn{ complement} of a graph $G$ is denoted by $\overline{G}$ and is the graph with vertex set $V(G)$ satisfying $uv\in E(\overline{G})$ if and only if $uv\notin E(G)$ for every $u,v\in V(G)$.
For a graph $G$ and positive integer $k$ the graph \defn{$kG$} is the graph obtained by taking $k$ disjoint copies of $G$ and making the copies anti-complete to each other.
%

\begin{proposition}\label{prop:containers1}
    There exists an integer $c$ with the following property.
    There exists an algorithm that takes as input $(G, \omega, k, b)$ where $\omega$ and $k$ are positive integers, $b \leq \omega$ is a non-negative integer, and $G$ is a $\overline{kK_\omega}$-free graph.
    There exists an algorithm that outputs a
    $(b, a)$-container family ${\cal F}$ in time $|\mathcal{F}| \cdot n^{O(1)}$ such that,
    \[
        |{\cal F}| \leq (n+1)^{  {(2\omega \cdot \log (n) + k+b)^{k+b+1} }}
        \qquad\text{and}\qquad
        a \leq  (2 \omega \cdot \log(n) + k+b)^{k+b+1}.
    \]
\end{proposition} 


\medskip
Given a graph $G$, a vertex subset $S\subseteq V(G)$, and a family $\calf$ of vertex subsets of $G$ that satisfy $V(G)\subseteq \bigcup_{F\in\calf} F$, we define the \defn{cover number} of $S$ using $\calf$, denoted as $\mathrm{cov}_{\mathcal{F}}(S)$, to be the minimum cardinality of a subset $\calf'\subseteq\calf$ which satisfies $S\subseteq \bigcup_{F\in\calf'}F$. Observe that $\mathrm{cov}_{\mathcal{F}}(S)\leq |S|$.

\begin{proposition}\label{prop:tw_and_talpha}
    Let $G$ be an $n$-vertex graph, let $f$, $a$, and $b$ be positive integers with $a \geq 2$, and let $\mathcal{F}$ be a $(b, a)$-container family in $G$. If $tw_\alpha(G) > 1020000 \cdot \lceil \log(2n) \cdot a^3 \cdot {\log(a\cdot f + 4)} \cdot f \rceil$, then there exists an induced subgraph $G' \subseteq G$ and an independent set $I \subseteq V(G')$ of size $680000 \cdot \lceil \log(2n) \cdot a^3 \cdot {\log(a\cdot f + 4)} \cdot f \rceil$ that satisfies the following properties:
    \begin{itemize}
        \item Every induced subgraph $H$ of $G'$ with $\alpha(H)\leq b$ satisfies
        $$|V(H)| \quad \leq \quad {42} \cdot b \cdot (\lceil\log 4|\mathcal{F}|\rceil + 680000 \cdot \lceil \log(2n) \cdot a^3 \cdot {\log(a\cdot f + 4)}\rceil).$$
        \item For every $(I,\frac{1}{2})$-balanced separator $S$ in $G'$, we have that $\mathrm{cov}_{\mathcal{F}}(S) \geq f$.
    \end{itemize}
\end{proposition}


\medskip
\noindent Now we are ready to prove the main structure theorem, which we re-state for convenience. 

\smallskip
\noindent
{\bf Theorem~\ref{thm:atw_bound}. }{\em 
For every positive integer $t$, there exists a positive integer $\nu'(t)$ and real $\epsilon(t) > 0$ such that, for every graph $G\in\calc_t$ on $n$ vertices,
\[
   \atw{G} \ \ \leq\ \ 2^{\nu'(t)\log^{1-\epsilon(t)}n}.
\]
}

\begin{proof}
    Let us assume without loss of generality that $n \geq 2$, since otherwise the theorem holds trivially.
    Let $G\in\calc_t$ be a graph on $n$ vertices.
    Set $a:=(2t\cdot \log n + 3)^4$ and let $c_1$ be the smallest constant such that $a\leq c_1\log^4 (n+1)$ for every non-negative integer $n$, and observe that $c_1$ depends only on $t$.
    Since $G\in\calc_t$, we have that $G$ is $\overline{2K_t}$-free, and so \Cref{prop:containers1} applied to the tuple $(G,t,2,1)$ implies the existence of a $(1,a)$-container family $\mathcal{F}$ in $G$ satisfying $|\mathcal{F}|\leq(n+1)^a$.
    Define $\hat\nu(t)$ to be the smallest positive integer $z$ such that, for every $n \geq 2$,
    \[
        (\omega_n(z))^{\mu(t)}\cdot f_n(\nu(t)) \ \ <\ \ f_n(z),
    \]
    where $\mu(t)$ and $\nu(t)$ are the constants in \Cref{lem:half_balanced_sep}, and $f_n(z)$ and $\omega_n(z)$ are functions on the set of positive integers, defined as follows:
    $$
    f_n(z) := 2^{z\log^{1-\eball/2}n}
    \ \ \text{and}\ \
    \omega_n(z) := 42\cdot(\ceil{a\cdot \log (n+1) + 2} + 680000\cdot\ceil{\log2n\cdot a^3\cdot\log(a\, f_n(z)+4)})
    $$

    \begin{claim}
        $\hat\nu(t)$ is well defined.
    \end{claim}
    \begin{claimproof}
        We need to show that the set of positive integers $z$ that satisfy $(\omega_n(z))^{\mu(t)}\cdot f_n(\nu(t)) \leq f_n(z)$ is non-empty. 
        First, observe that since $n\geq 2$,
        \begin{align*}
            \omega_n(z)\quad  
            &=\quad 42\cdot\Bigl(\ceil{a\cdot\log(n+1) + 2}\
                +\ 680000\cdot\ceil{\log(2n)\cdot a^3\cdot\log(a\, f_n(z)+4)}\Bigr) \\
            &\leq\quad 42\cdot\ceil{c_1\log^5(n+1) + 2} \\
            &\hspace{50pt}
                +\ 42\cdot 680000\cdot\ceil{\log(2n)\cdot c_1^3\cdot\log^{12}(n+1)\cdot
                  \log(c_1\log^4(n+1)\cdot 2^{z\log^{1-\eball/2}n} + 4)} \\
            &\leq\quad c_2\log^5(n+1)
                +\ c_3\log^{13}(n+1)\cdot\bigl(\log c_1 + 4\log\log(n+1)
                  + z\log^{1-\eball/2}n + 2\bigr) \\[2pt]
            &\leq\quad c_2\log^5(n+1)\; +\; c_3\log^{13}(n+1)\cdot\bigl(c_4\cdot z\cdot \log n \bigr) \\[2pt]
            &\leq\quad c_5\cdot z\cdot \log^{14}n
        \end{align*}
        Here $c_2,c_3,c_4$ and $c_5$ are suitably chosen constants that depend only on $t$.
        Now, using this inequality we obtain the following upper bound: 
        \[
        (\omega_n)^{\mu(t)}\cdot f_n(\nu(t)) \
        \leq\ \left(c_5\cdot z\cdot \log^{14}n\right)^{\mu(t)} 2^{\nu(t)\log^{1-\eball/2}n} \
        \leq\ 2^{\mu(t)(\log c_5\, +\, 14\log\log n\, +\, \log z)\, +\, \nu(t)\log^{1-\eball/2}n}
        \]
        Since $O(\log\log n)\subseteq o(\log^{1-\eball/2}n)$ and since $O(\log z)\subseteq o(z)$, we conclude that there exists a large enough positive integer $z$ such that, for every positive integer $n\geq 2$,
        $$
        \mu(t)\,(\log c_5\; +\; 14\log\log n\; +\; \log z)\ +\ \nu(t)\log^{1-\eball/2}n\quad \leq\quad z\log^{1-\eball/2}n.
        $$
        Thus, the set of positive integers $z$ such that, $(\omega_n(z))^{\mu(t)}\cdot f_n(\nu(t)) < f_n(z)$ holds for every positive integer $n\geq 2$, is non-empty.
        Consequently $\hat\nu(t)$ is well defined.
    \end{claimproof}
    
    Let us denote $\omega_n(\hat\nu(t))$ as $\omega_n$ and $f_n(\hat\nu(t))$ as $f_n$.
    Now, suppose for contradiction that $\atw{G} > 1020000\cdot\ceil{\log(2n)\cdot a^3\cdot\log(a\, f_n+4)\cdot f_n}$.
    Applying \Cref{prop:tw_and_talpha} to the tuple $(G,f_n,a,1,\mathcal{F})$, we obtain an induced subgraph $G'\subseteq G$ and an independent set $I\subseteq V(G')$, such that every induced subgraph $H$ of $G'$ with $\alpha(H)\leq 1$ satisfies $|V(H)|\leq \omega_n$, and every $(I,\frac{1}{2})$-balanced separator $S$ in $G'$ satisfies $\mathrm{cov}_\mathcal{F}(S)\geq f_n$.
    Since a maximum clique $H$ of $G'$ satisfies $\alpha(H)\leq1$ and $|V(H)|=\omega(G')$, the former property gives $\omega(G')\leq\omega_n$.
    Furthermore, $G'\in\calc_t$ since $\calc_t$ is hereditary. 
    Thus, applying \Cref{lem:half_balanced_sep} to $G'$ yields an $(I,\tfrac{1}{2})$-balanced separator $S$ of size at most $\omega(G')^{\mu}\cdot f_{n'}(\nu(t))$, which is at most $\omega_n^{\mu}\cdot f_n(\nu(t))< f_n$.
    As $\mathcal{F}$ is a $(1,a)$-container family, this implies $\mathrm{cov}_\mathcal{F}(S)\leq|S|<f_n$, which is a contradiction.
    Thus, we get,
   \begin{align*}
        \atw{G} \ \
        &\leq\ \ 1020000\cdot\ceil{\log(2n)\cdot a^3\cdot\log(af_n+4)\cdot f_n} \\
        &=\ \ 1020000\cdot\Bigl\lceil
            \log(2n)\cdot c_1^3\cdot\log^{12}(n+1)\cdot 
            \log\bigl(c_1\cdot \log^4(n+1)\cdot 2^{\hat\nu(t)\log^{1-\eball/2}n} + 4\bigr)
            \cdot f_n \Bigr\rceil\\
        &\leq\ \ c_2\cdot \log^{13} n\cdot 
            \Bigl(\log c_1 + 4\log\log (n+1) + \hat\nu(t)\log^{1-\eball/2}n + 2\Bigr)
            \cdot f_n \\
        &\leq\ \ c_3\cdot \log^{14} n\cdot 2^{\hat\nu(t)\log^{1-\eball/2}n}\\
        &\leq\ \ 2^{\log(c_3)\, +\, 14\log\log n\, +\, \hat\nu(t)\log^{1-\eball/2}n}
    \end{align*}
    Since $O(\log\log n)\subseteq o(\log^{1-\eball/2}n)$, we conclude that there exists some constant $\nu'(t)$ depending only on $t$ such that, for every $n\geq 2$,
    $$
    \log c_3\ +\ 14\log\log n\ +\ \hat\nu(t)\log^{1-\eball/2}n \quad \leq \quad \nu'(t)\log^{1-\eball/2}n
    $$ 
    Thus, setting $\nu'(t)$ as above and $\epsilon(t) = \eball/2$ completes the proof of the theorem.
\end{proof}

\section{Tree Independence of Biclique-Induced-Minor-Free Graphs}\label{sec:8}
In this section, we prove Theorem~\ref{thm:biclique}. Towards the proof, we need a few ingredients.




We also rely on the following separator theorem of Korhonen and
Lokshtanov~\cite{korhonen2024induced}.
They state and prove it for an arbitrary pattern graph $H$ excluded as an induced minor, we re-state here their separator theorem as it applies to $H=K_{t,t}$.

\begin{proposition}[\cite{korhonen2024induced}]\label{prop:km_separator}
    There is a randomized polynomial-time
    algorithm that, given a graph $G$ and positive integer $t$, outputs either an induced minor model of
    $K_{t,t}$ in $G$, or a $(V(G),\tfrac23)$-balanced separator of $G$ of size at
    most $O\bigl(t^3\sqrt{|E(G)|}\bigr)$.
\end{proposition}

The next lemma upgrades \Cref{prop:km_separator} from a balanced separator to a
$(Y,\tfrac12)$-balanced separator for a prescribed set $Y$, at the cost of a
$\log|V(G)|$ factor in the size.

\begin{lemma}\label{lem:Y_balanced_sep}
    Let $t$ be a positive integer, let $G$ be a $K_{t,t}$-induced-minor-free
    graph, and let $Y\subseteq V(G)$.
    Then $G$ has a $(Y,\tfrac12)$-balanced separator of size at most
    $O\bigl(t^3\log|V(G)|\cdot\sqrt{|E(G)|}\bigr)$.
\end{lemma}
\begin{proof}
    Write $n:=|V(G)|$ and let $b:=O\bigl(t^3\sqrt{|E(G)|}\bigr)$ denote the size
    bound of \Cref{prop:km_separator} for $G$.
    If $|Y|\le1$, then $S:=Y$ is a $(Y,\tfrac12)$-balanced separator of size at
    most one, so assume $|Y|\ge2$.
    
    We prove by induction on |C| that, for every $C\subseteq V(G)$ there is a set $S\subseteq C$ with $|S|\le b\bigl(\log_{3/2}|C|+1\bigr)$ such that every connected component of $G[C]-S$ contains at most $\tfrac12|Y|$ vertices of $Y$. The lemma then follows from this statement by taking
    $C:=V(G)$, since $b\bigl(\log_{3/2}n+1\bigr)=O(b\log n)=
    O\bigl(t^3\log n\cdot\sqrt{|E(G)|}\bigr)$.


        We argue by induction on $|C|$.
        If no connected component of $G[C]$ contains more than $\tfrac12|Y|$
        vertices of $Y$, then $S:=\emptyset$ suffices.
        Otherwise \Cref{prop:km_separator} provides a
        $(V(C),\tfrac23)$-balanced separator $S_1$ of $G[C]$ with $|S_1|\le b$.
        %
        Every connected component of $G[C]-S_1$ has at most $\tfrac23|C|$
        vertices, and since these components partition $Y\cap C$, at most one of
        them contains more than $\tfrac12|Y|$ vertices of $Y$.
        If there is no such component, then $S:=S_1$ suffices, as
        $|S_1|\le b\le b(\log_{3/2}|C|+1)$.
        Otherwise let $C'$ be the unique such component, so $|C'|\le\tfrac23|C|$
        and in particular $|C'|<|C|$.
        By the induction hypothesis applied to $C'$ there is a set
        $S_2\subseteq C'$ with 
        $$|S_2|\le b(\log_{3/2}|C'|+1)\le
        b(\log_{3/2}(\tfrac23|C|)+1)=b\log_{3/2}|C|$$ 
        such that every component of
        $G[C']-S_2$ contains at most $\tfrac12|Y|$ vertices of $Y$.
        Set $S:=S_1\cup S_2$.
        Every connected
        component of $G[C]-S$ either is a component of $G[C]-S_1$ other than
        $C'$, or is a component of $G[C']-S_2$; in both cases it contains at most
        $\tfrac12|Y|$ vertices of $Y$.
        Finally $|S|\le|S_1|+|S_2|\le b+b\log_{3/2}|C|=b(\log_{3/2}|C|+1)$,
        completing the induction.
\end{proof}

Finally, we use the following degree-boundedness result of Bourneuf, Buci\'c,
Cook, and Davies~\cite{bourneuf2024polynomial}.

\begin{proposition}[\cite{bourneuf2024polynomial}]\label{prop:degree_bounded}
    For every positive integer $t$ there exist $c,d>0$ such that for every
    positive integer $\omega$, every graph $G$ that contains no induced
    subdivision of $K_{t,t}$ and no clique on more than $\omega$ vertices has average degree
    at most $c\,\omega^{d}$.
\end{proposition}

We are now ready to prove \Cref{thm:biclique}, which we re-state here for convenience.

\smallskip
\noindent
{\bf Theorem~\ref{thm:biclique}. }{\em 
For every positive integer $t$ there exists a positive integer $d$ such that for every $K_{t,t}$-induced-minor-free graph $G$, $\mathrm{tw}_\alpha(G) \leq O(\sqrt{n}(\log n)^d)$.
}

\begin{proof}[Proof of \Cref{thm:biclique}]
    Let $G$ be a $K_{t,t}$-induced-minor-free graph and set $n:=|V(G)|$.
    We may assume that $n$ is larger than some constant $n_0=n_0(t)$: for smaller
    $n$ we have $\mathrm{tw}_\alpha(G)\le n=O(1)$, so the claimed bound holds by
    taking the constant in the $O$-notation large enough.
    Throughout, asymptotic notation hides constants depending only on $t$.

    Since $\overline{2K_t}=K_{t,t}$ and $G$ contains no induced $K_{t,t}$, $G$ is $\overline{2K_t}$-free.
    Applying \Cref{prop:containers1} to the tuple $(G,t,2,1)$ therefore yields a
    $(1,a)$-container family $\mathcal F$ of $G$ with
    $a\ \le\ (2t\log n+3)^4\ =\ O(\log^4 n)$
    and
    $|\mathcal F|\ \le\ (n+1)^a$.
    Since a $(1,1)$-container family is also a $(1,2)$ container family without loss of generality  $a \geq 2$. Furthermore, $\log|\mathcal F|\le a\log(n+1)=O(\log^5 n)$.

    Let $c,d_0>0$ be the constants provided by \Cref{prop:degree_bounded} for
    $t$, and set
    \[
        f\ :=\ \bigl\lceil \sqrt n\,(\log n)^{7d_0+2}\bigr\rceil .
    \]
    Suppose towards a contradiction that
    \[
        \mathrm{tw}_\alpha(G)\ >\ 1020000\,
        \bigl\lceil \log 2n\cdot a^3\cdot\log(af+4)\cdot f\bigr\rceil .
    \]
    Since $a\ge2$, \Cref{prop:tw_and_talpha} applied to $(G,f,a,1,\mathcal F)$
    produces an induced subgraph $G'\subseteq G$ and an independent set
    $I\subseteq V(G')$ such that
    \begin{itemize}
        \item[(i)] every clique $H$ of $G'$
        satisfies
        $|V(H)|\le 42\bigl(\lceil\log 4|\mathcal F|\rceil
        +680000\lceil\log 2n\cdot a^3\cdot\log(af+4)\rceil\bigr)$; and
        \item[(ii)] every $(I,\tfrac12)$-balanced separator $S$ of $G'$ satisfies
        $\mathrm{cov}_{\mathcal F}(S)\ge f$.
    \end{itemize}

We show that the clique number of $G'$ is polylogarithmic. Let $H$ be the largest clique in $G'$, and set $\omega = |H|$.
    As $f\le\sqrt n(\log n)^{7d_0+2}+1\le n$ for $n\ge n_0$, we have
    $\log(af+4)=O(\log n)$. By~(i) we have that
    \[
        \omega \le\ 42\bigl(\lceil\log 4|\mathcal F|\rceil
        +680000\lceil\log 2n\cdot a^3\cdot\log(af+4)\rceil\bigr)
        \ =\ O(\log^5 n)+O(\log^{14} n)\ =\ O(\log^{14} n),
    \]
    where we used $a^3=O(\log^{12}n)$ and $\log 2n$ and $\log(af+4)$ are both upper bounded by $O(\log n)$. 
    
Since $G'$ contains no induced subdivision of $K_{t,t}$ and has no clique on more than $\omega$ vertices, \Cref{prop:degree_bounded}
    bounds the average degree of $G'$ by $c\omega^{d_0}=O(\log^{14 d_0}n)$.
    Therefore $|E(G')| = O(n\log^{14 d_0}n)$, so $\sqrt{|E(G')|} = O(\sqrt n\,\log^{7d_0}n)$.
Hence, applying \Cref{lem:Y_balanced_sep} to $t$, the ($K_{t,t}$-induced-minor-free)
    graph $G'$, and $Y:=I$ yields an $(I,\tfrac12)$-balanced separator $S$ of $G'$
    with
    \[
        |S|\ =\ O\bigl(t^3\log|V(G')|\cdot\sqrt{|E(G')|}\bigr)
        \ =\ O\bigl(\log n\cdot\sqrt n\,\log^{7d_0}n\bigr)
        \ =\ O\bigl(\sqrt n\,\log^{7d_0+1}n\bigr).
    \]
On the other hand every vertex in $V(G)$ is contained in some set in $\mathcal F$ (since it is a clique), thus $\mathrm{cov}_{\mathcal F}(S)\le|S|$, and combining this with~(ii) gives
    $|S|\ge\mathrm{cov}_{\mathcal F}(S)\ge f=\bigl\lceil\sqrt n(\log n)^{7d_0+2}
    \bigr\rceil$.
    But $O(\sqrt n\,\log^{7d_0+1}n)<\sqrt n(\log n)^{7d_0+2}\le f$ for $n\ge n_0$,
    a contradiction.

We conclude that 
\begin{align*}
   \mathrm{tw}_\alpha(G) & \leq 1020000 \cdot \bigl\lceil \log 2n\cdot a^3\cdot\log(af+4)\cdot f\bigr\rceil = O\bigl(\log n\cdot\log^{12}n\cdot\log n\cdot\sqrt n\,\log^{7d_0+2}n\bigr) \\
   & = O\bigl(\sqrt n\,(\log n)^{7d_0+16}\bigr).
\end{align*}
The theorem follows with $d:=7d_0+16$.
\end{proof}

\section{Conclusion}\label{sec:conclusion}
We have shown that for every induced-minor-closed class ${\cal H}$ the tree-independence of ${\cal H}$ grows linearly in the number of vertices, or on the order of the square root of number of vertices, or is sub-polynomial in the number of vertices. 
The tree-independence is sub-polynomial whenever ${\cal H}$ excludes at least one complete bipartite graph and at least one wall. 
Our main theorem is a partial resolution of the conjecture of Chudnovsky et al.~\cite{alphacontainers} that for every integer $t$ the graphs that exclude the complete bipartite graph $K_{t,t}$ and the wall $W_{t\times t}$ as an induced minor have poly-logarithmic tree-independence number. A full resolution of this conjecture remains elusive. 

\paragraph{AI Disclosure.}
We used Claude (Anthropic) and ChatGPT (OpenAI)  to assist with
(i)~searching for relevant literature,
(ii)~tightening and improving the language and exposition throughout the paper,
(iii)~suggesting formulations for parts of the introduction, 
(iv)~catching and proposing fixes for some small mistakes 
(v)~generating initial draft for proofs of individual Lemmas / Theorems based on proof strategies and technical direction from the authors (Sections~\ref{sec:absep} and~\ref{sec:8} only).

The tool materially affected the exposition across all sections of the paper.
The authors have verified the correctness and originality of all content, including all mathematical statements, proofs, and references, and take complete responsibility for the same.

\bibliographystyle{plainurl}
\bibliography{my_bib}

\end{document}